\documentclass[a4paper,11pt]{amsart}
\usepackage{amsmath,amsfonts,amsthm,amssymb,enumerate, comment}
\usepackage{graphicx}

\numberwithin{equation}{section}

\newtheorem{theorem}{Theorem}[section]

\newtheorem{proposition}[theorem]{Proposition}

\newtheorem{lemma}[theorem]{Lemma}

{
\theoremstyle{definition}
\newtheorem{definition}[theorem]{Definition}
\newtheorem{example}[theorem]{Example}
\newtheorem{remark}[theorem]{Remark}
}

\title{Transformations of lattice diagrams and their associated dotted diagrams}
\author{Inasa Nakamura}
\address{Department of Mathematics, Information Science and  Engineering, \newline
Saga University, \newline 
1 Honjomachi, Saga, 840-1153, Japan}
\email{inasa@cc.saga-u.ac.jp}

\begin{document}  
\begin{abstract}
We consider a graph called a lattice diagram, which is a graph in the $xy$-plane such that each edge is parallel to the $x$-axis or the $y$-axis. In \cite{N}, we investigated transformations of certain lattice diagrams, and we considered the reduced diagram that is obtained from deformations of a diagram  associated with a lattice diagram. In this paper, we refine the notion of the reduced diagram by introducing the notion of a dotted diagram. A lattice diagram is presented by an admissible dotted diagram.
 We investigate deformations of dotted diagrams, and we investigate relation between deformations of admissible dotted diagrams and transformations of lattice diagrams, giving results that are refined and corrected versions of \cite[Lemma 6.2, Theorem 6.3]{N}. 
\end{abstract}
\maketitle

\section{Introduction}\label{sec1}

Let $\mathbb{R}^2$ be the $xy$-plane. We say that a segment is in the {\it $x$-direction} or in the {\it $y$-direction} if it is parallel to the $x$-axis or the $y$-axis, respectively.

\begin{definition}
A {\it lattice diagram} is a finite graph $P$ in $\mathbb{R}^2$ consisting of a finite number of vertices of degree zero with multiplicity $2$, vertices of degree $2$ and $4$ with multiplicity $1$, and edges with multiplicity $1$, satisfying the following conditions. 

\begin{enumerate}[$(1)$]
\item
Each edge is either in the $x$-direction or in the $y$-direction. 

\item
The $x$-components (respectively $y$-components) of isolated vertices and edges in the $y$-direction (respectively $x$-direction)  are mutually distinct except for pairs of diagonal edges around a vertex of degree 4. 

\item
Around each vertex of degree $2$, the pair of edges have a coherent orientation. 

\item
Around each vertex of degree $4$, 
each pair of diagonal edges have a coherent orientation. 
\end{enumerate}

We call a vertex of degree zero (respectively, $4$) an {\it isolated vertex} (respectively, a {\it crossing}). 
We denote by $\partial P$ the oriented graph obtained from $P$ by deleting isolated vertices, called the {\it essential diagram} of $P$. 

The set of vertices are divided into two sets $X_0$ and $X_1$ 
such that each edge in the $x$-direction is oriented from a vertex of $X_0$ to a vertex of $X_1$, and isolated vertices are both in $X_0$ and $X_1$ with multiplicity 1. We denote $X_0$ (respectively $X_1$) by $\mathrm{Ver}_0(P)$ (respectively $\mathrm{Ver}_1(P)$), and we call it the set of {\it initial vertices} (respectively {\it terminal vertices}). 
When we draw figures, we denote an initial vertex (respectively a terminal vertex) by a small black disk (respectively an $X$ mark). And in our figures, the $x$-direction is the vertical direction, in order to be coherent with figures in \cite{N}. 
\end{definition}
\begin{figure}[ht]
\centering
\includegraphics*[height=4cm]{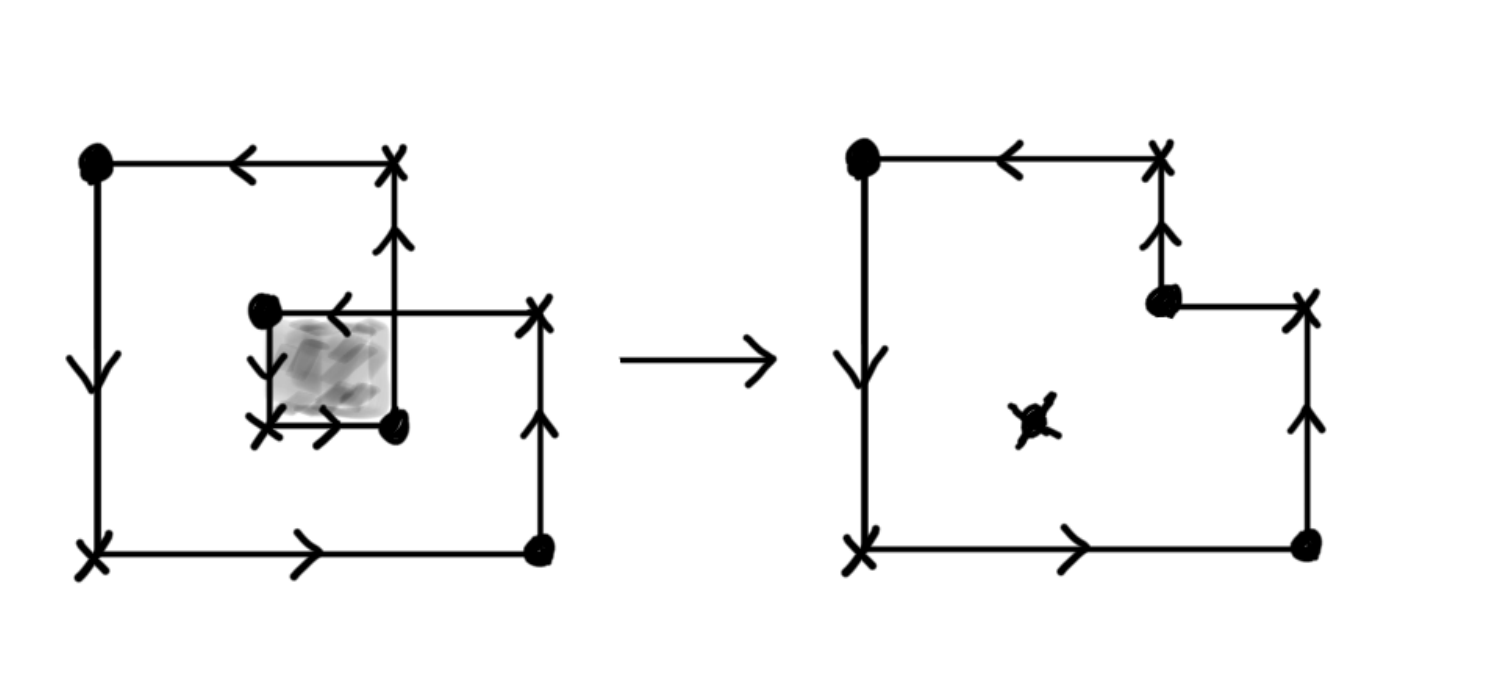}
\caption{A lattice diagram (left figure) and the result of a transformation along a rectangle (right figure). The rectangle is denoted by the shadowed area. We remark that the $x$-direction is the vertical direction, and we denote an initial vertex (respectively a terminal vertex) by a small black disk (respectively an $X$ mark). }
\label{Fig1}
\end{figure}

 See the left figure in Figure \ref{Fig1} for an example of a lattice diagram. 
In \cite{N}, we investigated partial matchings, using certain lattice diagrams (in \cite{N} called lattice polytopes) whose vertices of degree 2 are lattice points. 
A partial matching is a finite graph consisting of edges and vertices of degree zero or one, with the vertex set $\{1,2,\ldots,m\}$ for some positive integer $m$; see \cite{Reidys} for researches on structures of partial matchings using chord diagrams. 
In \cite{N}, with a motivation to investigate partial matchings, we introduced the lattice presentation of a partial matching, and the lattice diagram associated with a pair of partial matchings. Further, in \cite{N}, we treated transformations of lattice presentations and lattice diagrams, and the area of a transformation, and the reduced diagram of a lattice diagram. A deformation of a lattice diagram is a sequence of transformations of lattice diagrams such that each transformation is described by a rectangle as in Figure \ref{Fig1}. 

In this paper, we refine the notion of a reduced diagram by introducing the notion called a dotted diagram (Definition \ref{def2-1}). 
A lattice diagram is presented by an admissible dotted diagram.
We introduce a refined version of deformations of dotted diagrams, and 
we show that a reduced diagram of a dotted diagram obtained by a certain sequence of deformations is uniquely determined up to certain local moves (Theorem \ref{prop3-5}), and that for a dotted diagram, any sequence consisting of certain deformations is finite (Theorem \ref{r-prop3-5}). These theorems provide a refined and corrected version of \cite[Lemma 6.2]{N}. 
Further, we give a refined and corrected version of \cite[Theorem 6.3]{N}: 
for an admissible dotted diagram $\Gamma$, its associated lattice diagram $P$ admits a dissolution with minimal area if it is deformed to the empty graph by specific deformations, and if $P$ admits a transformation with minimal area, then $\Gamma$ is deformed to the empty graph by certain deformations (Theorem \ref{thm3-7}); see Remark \ref{rem0904}. Further, we show that if a dotted diagram has \lq\lq many'' dots, then it is admissible (Proposition \ref{prop3-8}), and a dotted diagram with ``many'' dots is reduced to the empty graph (Proposition \ref{prop3-9}).

The paper is organized as follows. In Section \ref{sec1}, we review transformations of lattice diagrams and the areas of transformations. 
In Section \ref{sec2}, we give the definitions of a dotted diagram and the deformations of dotted diagrams, and Theorems \ref{prop3-5} and \ref{r-prop3-5}. 
In section \ref{sec-proof}, we show Theorems \ref{prop3-5} and \ref{r-prop3-5}. 
In section \ref{deformation}, we investigate relation between deformations of admissible dotted diagrams and dissolutions of lattice diagrams, and we give Theorem \ref{thm3-7} and Propositions \ref{prop3-8} and \ref{prop3-9}. 
Section \ref{sec-lemma} is devoted to showing lemmas. 
We revise terminologies as shown in Table \ref{terminology}.

\begin{table}[h]
  \begin{center}
   
    \begin{tabular}{c|c}\hline\hline 
      In \cite{N}  & In this paper \\
      \hline
      lattice polytope $P$ & lattice diagram $P$ \\
      boundary $\partial P$ of $P$ & essential diagram $\partial P$ of $P$ \\
      transformation of $P$ along a rectangle   & transformation of $P$ along a rectangle\\
      transformation of $P$   & dissolution of $P$ \\
      dotted graph & dotted diagram \\
       reduced graph & reduced diagram \\
           \hline
    \end{tabular}
  \end{center}
   \caption[]{Terminologies}  
	\label{terminology}  
\end{table}

\section{Transformations of lattice diagrams and the areas of  transformations}\label{sec1}
In this section, we review transformations of lattice diagrams and the areas of transformations \cite{N}. 
\subsection{Transformations of lattice diagrams}
  
For a point $v=(x_1,y_1)$ of $\mathbb{R}^2$, we call $x_1$ (respectively $y_1$) the {\it $x$-component} (respectively the {\it $y$-component}) of $v$. 
For a set of points of $\mathbb{R}^2$, $\{ v_1, \ldots, v_n\}$  with $v_j=(x_j, y_j)$ ($j=1,\dots, n$), we call the set $\{ x_1, \ldots, x_n, y_1,\ldots, y_n\}$ the {\it set of $x$, $y$ components} of $\{ v_1,\ldots, v_n\}$. 

\begin{definition}
Let $\Delta$ be a set of points in $\mathbb{R}^2$. For a point $u$ of $\mathbb{R}^2$, we denote by $x(u)$ and $y(u)$ the $x$ and $y$-components of $u$, respectively.  For two distinct points $v, w$ of $\Delta$, 
let $R(v, w)$ be the rectangle one pair of whose diagonal vertices are $v$ and $w$. Put $\tilde{v}=(x(w), y(v))$ and $\tilde{w}=(x(v), y(w))$, which form the other pair of diagonal vertices of $R(v,w)$. Then, consider a new set of points obtained from $\Delta$, from removing $v,w$ and adding $\tilde{v}, \tilde{w}$. 
We call the new set of points the result of the {\it transformation of $\Delta$ by the rectangle} $R(v, w)$, denoted by $t( \Delta; R(v, w))$. 

For two sets of points $\Delta$ and $\Delta'$ in $\mathbb{R}^2$ with the same set of $x, y$-components, we define a {\it transformation from $\Delta$ to $\Delta'$}, denoted by $\Delta \to \Delta'$, as a sequence of transformations by rectangles such that the initial and the terminal points are $\Delta$ and $\Delta'$, respectively. 

For a lattice diagram $P$, we define a {\it dissolution of $P$} as a transformation from $\mathrm{Ver}_0(P)$ to $\mathrm{Ver}_1(P)$. We recall that $\mathrm{Ver}_0(P)$ and $\mathrm{Ver}_1(P)$ are the set of initial vertices and terminal vertices of $P$, respectively. 

For a lattice diagram $P'$, let $R$ be a rectangle that contains vertices of $\mathrm{Ver}_0(P')$ as a diagonal pair. 
We define the {\it transformation of $P'$ along the rectangle $R$} as 
the transformation from $P'$ to the lattice diagram whose initial vertices are $t(\mathrm{Ver}_0(P'); R)$ and terminal vertices are $\mathrm{Ver}_1(P')$, respectively. 
\end{definition}

A dissolution of a lattice diagram $P$ is described as a sequence of transformations of lattice diagrams along rectangles such that each rectangle contains initial vertices of each lattice diagram as a diagonal pair. See Figure \ref{Fig1} for an example of a transformation of a lattice diagram along a rectangle.

\subsection{Areas of transformations}\label{sec4-2}

For a lattice diagram $P$, we define the area of $P$ and the area of $|P|$, denoted by $\mathrm{Area}(P)$ and $\mathrm{Area}|P|$, respectively, as follows. 
Let $\mathcal{C}$ be a union of a finite number of immersed circles in $\mathbb{R}^2$. We call a connected component of the complement $\mathbb{R}^2 \backslash \mathcal{C}$ 
a {\it region} of $\mathcal{C}$. 
 
Let $f$ be a map from $S^1=\mathbb{R}/2\pi\mathbb{Z}$ to $\mathbb{R}/2\pi\mathbb{Z}$. Let $F: \mathbb{R} \to \mathbb{R}$ be a lift of $f$. The value $(F(x)-x)/2\pi$ does not depend on the choice of $F$ and $x$, which is called the {\it rotation number} of $f$. 
Let $C$ be an immersed oriented circle in $\mathbb{R}^2$. Let $A$ be a region of $C$.  We define the {\it rotation number of $C$ with respect to $A$} as the rotation number of the map $f: C=\mathbb{R}/2\pi \mathbb{Z} \to \mathbb{R}/2\pi\mathbb{Z}$ which maps $x \in C$ to the argument of the vector from a fixed interior point of $A$ to $x$, which is well-defined. 
For a union $\mathcal{C}$ of a finite number of immersed oriented circles in $\mathbb{R}^2$ and a region $A$ of $\mathcal{C}$, we define the {\it rotation number} of $\mathcal{C}$ with respect to $A$ as the sum of the rotation numbers of immersed oriented circles of $\mathcal{C}$, with respect to $A$.

Recall that for a lattice diagram $P$, 
we give an orientation of $\partial P$ by giving each edge in the $x$-direction (respectively in the $y$-direction) the orientation from a vertex of $\mathrm{Ver}_0(P)$ to a vertex of $\mathrm{Ver}_1(P)$ (respectively from a vertex of $\mathrm{Ver}_1(P)$ to a vertex of $\mathrm{Ver}_0(P)$). Then, by smoothing the angles, we regard $\partial P$ as an immersion of a union of several oriented circles equipped with vertices of degree 2. 

Let $A_1,\ldots, A_m$ be the regions of $\partial P$. For each region $A_i$, let $\omega(A_i)$ be the rotation number of $P$ with respect to $A_i$ ($i=1, \ldots, m$). 
Then we define $\mathrm{Area}_+(A_i)>0$ for a region $A_i$ by the area induced from $\mathrm{Area}_+(B)=|x_2-x_1||y_2-y_1|$ for a rectangle $B$ whose diagonal vertices are $(x_1, y_1)$ and  $(x_2, y_2)$. More precisely, we decompose $\mathrm{Cl}(A_i)$ as the union of rectangles $B_j$ such that $\mathrm{Cl}(A_i)=\cup_jB_j$ and $\mathrm{Int}(B_j) \cap \mathrm{Int}(B_k)=\emptyset$ for $j \neq k$, and we define $\mathrm{Area}_+(A_i):=\sum_j \mathrm{Area}_+(B_j)$. 
The area of $A_i$ does not depend on the choice of decomposition.  We remark that $B$ does not have the structure as a lattice diagram. 
Then, we define the {\it area} of $P$, denoted by $\mathrm{Area}(P)$, by $\mathrm{Area}(P)=\sum_{i=1}^m \omega(A_i)\mathrm{Area}_+(A_i)$, and the {\it area} of $|P|$, denoted by $\mathrm{Area}|P|$, by $\mathrm{Area}|P|=\sum_{i=1}^m |\omega(A_i)|\mathrm{Area}_+(A_i)$.

\begin{definition}

Let $P$ be a lattice diagram. We consider a dissolution of $P$,  
$\mathrm{Ver}_0(P)=\Delta_0 \to \Delta_1\to \cdots \to \Delta_k=\mathrm{Ver}_1(P)$, such that each $\Delta_{j-1} \to \Delta_j$ is a transformation by a rectangle $R_j$ 
($j=1,2,\ldots,k$). We regard each $R_j$ as a lattice diagram, whose initial vertices are those coming from $P$. 
Then, we call $\sum_{j=1}^k |\mathrm{Area}(R_j)|$ the {\it area of the dissolution of $P$}.

\end{definition}

Theorem 5.9 in \cite{N} implies the following. 

\begin{theorem}\label{thm3-10}
Let $P$ be a lattice diagram. We consider a dissolution of $P$,  
$\mathrm{Ver}_0(P)=\Delta_0 \to \Delta_1\to \cdots \to \Delta_k=\mathrm{Ver}_1(P)$, such that each $\Delta_{j-1} \to \Delta_j$ is a transformation by a rectangle $R_j$ 
($j=1,2,\ldots, k)$.  We regard each $R_j$ as a lattice diagram, whose initial vertices are those coming from $P$. 

Then
\begin{equation}\label{eq0}
\sum _{j=1}^k \mathrm{Area}(R_j) =\mathrm{Area}(P) 
\end{equation}
and 
\begin{equation}\label{eq2}
\sum_{j=1}^k |\mathrm{Area}(R_j)|\geq \mathrm{Area}|P|. 
\end{equation}

Further, when $P$ satisfies either the condition $(1)$ or $(2)$ in \cite[Theorem 5.9]{N}, there exist dissolutions which realize the equality of $(\ref{eq2})$.  

 \end{theorem}

We call a dissolution which realizes the equality of $(\ref{eq2})$ a dissolution {\it with minimal area}.  

\section{Dotted diagrams and their reduced diagrams}\label{sec2}
In this section, we introduce the notion of a reduced diagram of a lattice diagram. And we give theorems that a reduced diagram of a dotted diagram obtained by a certain sequence of deformations is uniquely determined up to certain local moves (Theorem \ref{prop3-5}), and that for a dotted diagram, any sequence consisting of certain deformations is finite (Theorem \ref{r-prop3-5}).

\begin{remark}\label{rem0904}

Theorems \ref{prop3-5} and \ref{r-prop3-5} provide a refined and corrected version of \cite[Lemma 6.2]{N}; we remark that the set of deformations given in \cite{N} does not include the cases that have \lq\lq overlapping layers'', and for deformations that delete circle/loop components, the condition for the label of the complement region is necessary to show the uniqueness; 
the original version of \cite[Lemma 6.2]{N} holds true when we add the condition of labels for deformations that delete circle/loop components, and any sequence of deformations satisfies the condition (A) given in Theorem \ref{prop3-5}. 

Theorem \ref{thm3-7} is a refined and corrected version of \cite[Theorem 6.3]{N}; we remark that \cite[Theorem 6.3]{N} does not hold for lattice diagrams that admit transformations other than those described by the given deformations. 
\end{remark}

\subsection{Dotted diagrams}

\begin{definition}\label{def2-1}
Let $\Gamma$ be a finite graph in $\mathbb{R}^2$ consisting of vertices of multiplicity $1$ and edges of multiplicity $1$ which are intervals smoothly embedded in $\mathbb{R}^2$. Then $\Gamma$ is a {\it dotted diagram} if each edge has an orientation and each vertex is of degree 2 or degree 4, satisfying the following conditions. 
\begin{enumerate}
\item
Around each vertex of degree 2, the edges has a coherent orientation. Further, the closure of the union of the edges is an interval smoothly embedded in $\mathbb{R}^2$. We denote the vertex by a small black disk, called a {\it dot}.  

\item
Around each vertex of degree 4, each pair of diagonal edges has a coherent orientation. Further, the closure of the union of the 4 edges is the union of a pair of smoothly embedded intervals in $\mathbb{R}^2$ with one transverse intersection point. We call the vertex a {\it crossing}.

\end{enumerate}

We regard edges connected by vertices of degree 2 as an edge  equipped with several dots. We call an edge or a part of an edge an {\it arc} (with/without dots). The complement of a dotted diagram $\Gamma \subset \mathbb{R}^2$ consists of a finite number of connected components. We call each component a {\it region}. We regard $\Gamma$ as a finite number of immersed oriented circles with transverse intersection points, and we assign each region with an integral label denoting the rotation number. 
We also call $\Gamma$ equipped with integral labels on regions a {\it dotted diagram}. 

Two dotted diagrams are the {\it same} if they are related by an ambient isotopy of $\mathbb{R}^2$. 

\end{definition}

Let $P$ be a lattice diagram. 
We consider the graph $\Gamma$ obtained from the essential diagram $\partial P$ of $P$ by changing the terminal vertices $\mathrm{Ver}_1(P)$ into inner points of edges: $\Gamma$ is in the form of $\partial P$, equipped with the initial vertices $\mathrm{Ver}_0(P)$, which are one of the two types of vertices of $P$: $\mathrm{Ver}_0(P)$ and $\mathrm{Ver}_1(P)$. We smooth the angles of $\Gamma$. 
Then $\Gamma$ is a dotted diagram, which will be called the {\it dotted diagram associated with a lattice diagram $P$}; see Figure \ref{Fig2}.
We say that a dotted diagram $\Gamma$ is {\it admissible} if there exists a lattice diagram $P$ with which $\Gamma$ is associated. 

\begin{figure}[ht]
\centering
\includegraphics*[height=4cm]{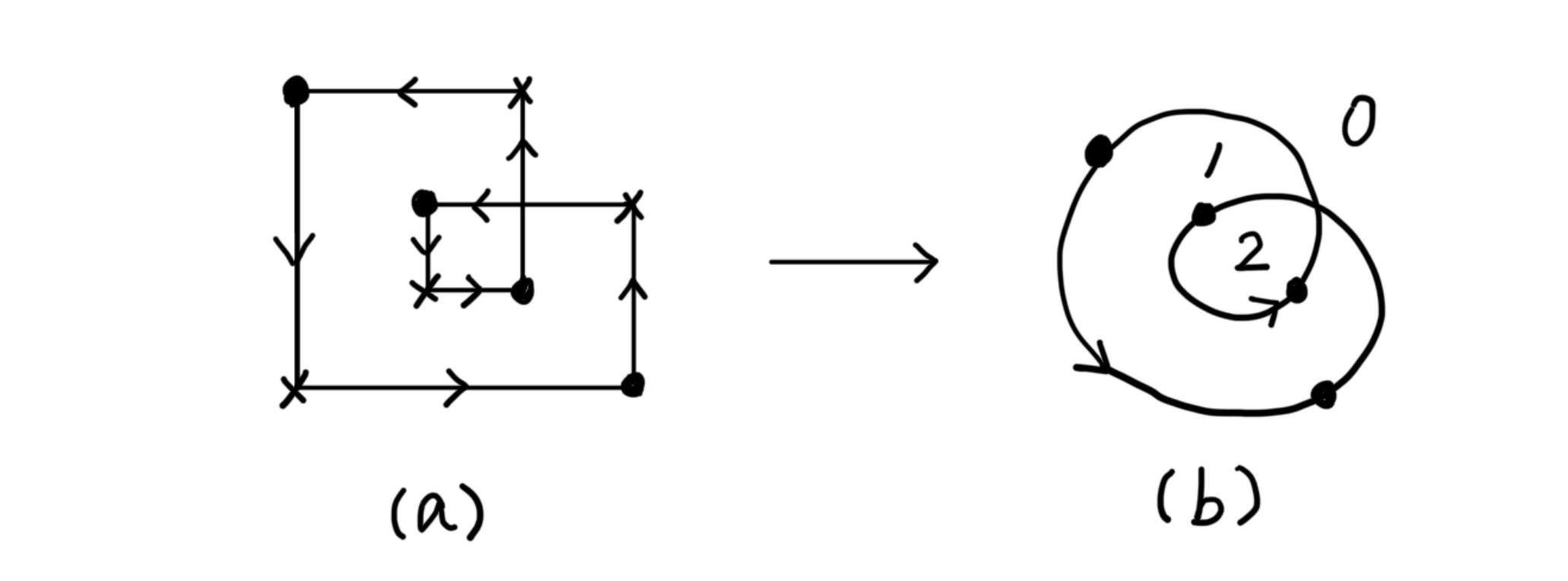}
\caption{(a) A lattice diagram and (b) the associated dotted diagram.}
\label{Fig2}
\end{figure}

Let $\Gamma$ be a dotted diagram. 
We call a set of arcs with dots and connecting crossings a {\it circle component} if it is the boundary of a smoothly embedded disk and any pair of arcs connected by a crossing is a pair of diagonal arcs. And we call a set of arcs with dots and connecting crossings a {\it loop component} if it is the boundary of an embedded disk $D$ with one non-smooth point and one pair of arcs connected by a crossing $c$ is a pair of adjacent arcs and 
any other pair of arcs connected by a crossing is a pair of diagonal arcs, and moreover around the crossing $c$ connecting the adjacent arcs forming the component, the disk $D$ does not contain the other two arcs. 

\subsection{Layer decomposition of a dotted diagram}

Let $\Gamma$ be a dotted diagram. 
Let $X$ be the closure of the union of a finite number of regions of $\Gamma$, where we ignore labels of the regions, equipped with dots on the boundary $\partial X$,  satisfying the followings. 
\begin{enumerate}
\item
The subset $X \subset \mathbb{R}^2$ is homeomorphic to an embedded disk. 
\item
The orientations of the edges of $\Gamma \cap \partial X$ induce a coherent orientation as a circle to the boundary $\partial X$. 
\item
The boundary $\partial X$ is equipped with dots induced from $\Gamma$. 
\end{enumerate}
We assign $X$ the sign $+1$ (respectively $-1$) if the induced orientation of $\partial X$ is anti-clockwise (respectively clockwise). 
We call $X$ a {\it layer}. In particular, we call $X$ 
a {\it positive layer} (respectively a {\it negative layer}) if the sign is $+1$ (respectively $-1$).

 When a dotted diagram $\Gamma$ is presented as the union $\cup_{j=1}^m \partial X_j$ of a finite number of layers $X_1, \ldots, X_m$ such that $\partial X_i \cap \partial X_j$ ($i \neq j$, $i,j \in \{1, \ldots, m\}$) consists of a finite number of points, we call 
the set $\{X_1, \ldots, X_m\}$ a {\it layer decomposition} of $\Gamma$. 
For a region $R$ of $\Gamma$, we call a layer $X$ a layer {\it forming $R$}  if $X \cap R \neq \emptyset$. 
For a region $R$ of $\Gamma$, its label is the sum of the signs of the layers forming $R$. 

\begin{example}
We consider a dotted diagram $\Gamma$ as illustrated in the top figure in Figure \ref{20240930-2}. The dotted diagram $\Gamma$ consists of two embedded circles, denoted by $C_1$ and $C_2$ which transversely intersect each other at two points. We denote by $D_1$ and $D_2$ the embedded disks whose boundaries are $C_1$ and $C_2$, respectively. Then, $\Gamma$ has two layer decompositions. One layer decomposition consists of two positive layers $D_1$ and $D_2$, both with the sign $+1$, and the other layer decomposition consists of two positive layers $D_1 \cup D_2$ and $D_1 \cap D_2$, both with the sign $+1$. 
\end{example}

\begin{figure}[ht]
\centering
\includegraphics*[height=6cm]{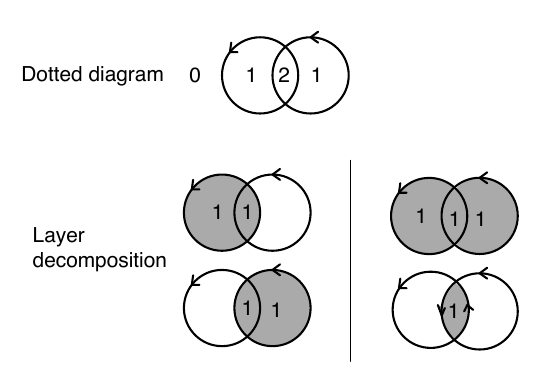}
\caption{Example of a dotted diagram and its layer decompositions, where each layer is the shadowed embedded disk.  We have two choices of layer decompositions.}\label{20240930-2}
\end{figure}

\begin{figure}[ht]
\centering
\includegraphics*[height=3cm]{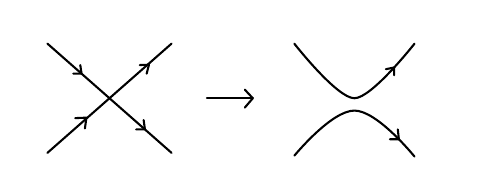}
\caption{Smoothing a crossing. }\label{20240930-1}
\end{figure}

\begin{proposition}
Any dotted diagram $\Gamma$ has a layer decomposition. 

\end{proposition}

\begin{proof}
Let $\Gamma_0$ be the dotted diagram obtained from $\Gamma$ by smoothing each crossing as depicted in Figure \ref{20240930-1}. The new dotted diagram consists of a finite number of embedded circles. 
For each embedded circle $C$, we assign to the embedded disk whose boundary is $C$ the sign $+1$ (respectively $-1$) if $C$ is oriented anti-clockwise (respectively, clockwise). Then, each circle $C$ induces a layer of $\Gamma$, such that the boundary of each layer, by modifying slightly at crossings, is $C$. 
By this method, $\Gamma_0$ induces a layer decomposition of $\Gamma$. 
\end{proof}

\subsection{Background and overlapping layers of $\Gamma \cap D$}

Let $\Gamma\neq \emptyset$ be a dotted diagram. 
We denote by $V_*$ the set of dots and crossings of $\Gamma$. 
Let $D$ be an embedded disk in $\mathbb{R}^2$ such that $\Gamma \cap D\neq \emptyset$ and $\Gamma \cap \partial D$ consists of a finite number of transverse intersection points of $\Gamma \backslash V_*$ and $\partial D$.

Assume that there exist dotted diagrams $G$ and $G'$ in $\mathbb{R}^2$ and layers $X_1, \ldots, X_m$ ($X_i \neq X_j$ for $i\neq j$) of $G'$  such that $\{X_1, \ldots, X_m\}$ is a subset of a layer decomposition of $G'$, satisfying the following. Let $V$ (respectively $V'$) be the set of dots and crossings of $G$ (respectively $G'$). Then, $G \cap G'$ consists of a finite number of transverse intersection points of $G\backslash V$ and $G' \backslash V'$, and $X_i \cap D\neq \emptyset$ for each $i=1, \ldots, m$, and $\Gamma \cap D$ is the closure of the union of $G \cap D$ and $\partial X_1\cap \mathrm{Int}(D), \ldots, \partial X_m\cap \mathrm{Int}(D)$. Then we call $G \cap D$ 
a {\it background} of $\Gamma \cap D$. 
We call  $X_1\cap D, \ldots, X_m\cap D$ {\it overlapping layers} of $\Gamma \cap D$, or layers {\it overlapping} the background $G \cap D$. 
For a region $R$ of $G$ that intersects with $D$, we call the non-empty intersection $R \cap D$ a {\it block} of $G$; the closure of a block is the closure of the intersection with $D$ of the union of several regions of $\Gamma$. 
 
The label of a region $R$ of $\Gamma$ is the sum of the label of the block of the background containing $R$ and the labels of the overlapping layers containing $R$. 

\begin{example}
We consider a dotted diagram $\Gamma$ and an embedded disk $D$ as illustrated in the top figure in Figure \ref{20240930-3}. Let $G$ and $G'$ be dotted diagrams as shown in the lower figures in Figure \ref{20240930-3}. The dotted diagram $G'$ has a layer decomposition $\{ X\}$ with $G'=\partial X$. We see that $\Gamma \cap D$ is the closure of $(G \cap D) \cup (\partial X \cap \mathrm{Int}(D))$. Hence $G \cap D$ is a background and $X \cap D$ is an overlapping layer of $\Gamma \cap D$. 
\end{example}

\begin{figure}[ht]
\centering
\includegraphics*[height=7cm]{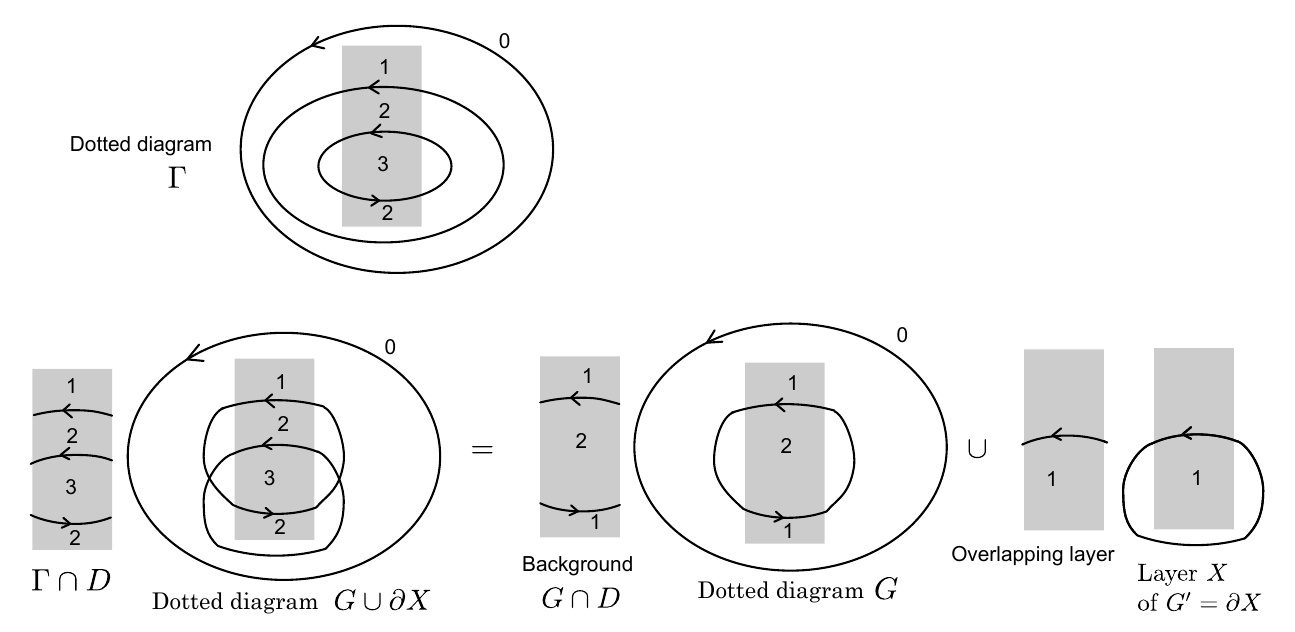}
\caption{Example of a dotted diagram $\Gamma$ and its intersection with a disk $D$. By taking a dotted diagram $G$ such that the intersection with $D$ of $G$ and its overlapping layer is $\Gamma \cap D$, we decompose $\Gamma \cap D$ into a background and an overlapping layer.  }
\label{20240930-3}
\end{figure}

\begin{definition}\label{def3-2}
Let $\Gamma$ and $\Gamma'$ be dotted diagrams. 
Let $V$ (respectively $V'$) be the set of dots and crossings of $\Gamma$ (respectively $\Gamma'$), respectively. 
Then, 
$\Gamma'$ is obtained from $\Gamma$ by 
a {\it dotted diagram deformation} or simply a  {\it deformation} if $\Gamma$ is changed to $\Gamma'$ by a local move in an embedded disk $D\subset \mathbb{R}^2$, satisfying the following. \\

\noindent
(1)\ 
The intersection $\Gamma \cap \partial D$ (respectively $\Gamma' \cap \partial D$) consists of transverse intersection points of $\Gamma\backslash V$ (respectively $\Gamma' \backslash V'$) and $\partial D$. \\

\noindent
(2)\ 
The dotted diagrams are identical in the complement of $D$: 
$\Gamma \cap (\mathbb{R}^2\backslash D)=\Gamma' \cap (\mathbb{R}^2\backslash D)$. \\

\noindent
(3)\ 
The intersection $\Gamma \cap D$, is deformed to $\Gamma' \cap D$, by one of the following.

 \begin{enumerate}[(I)]
 \item
Reduce several dots on an arc to one dot on the arc; see Figure \ref{Fig3} (I).\\

For (II) (III) (IV), we consider the situation that $\Gamma \cap D$ consists of a background $G\cap D$ with overlapping layers such that overlapping layers have the same sign $\epsilon$ ($\epsilon \in \{+1, -1\}$). We apply deformations II, III, IV when $\Gamma \cap D$ satisfies the assumption as follows. We remark that $\Gamma\cap D$ may have choices of the background $G\cap D$. We call the deformations {\it deformations II, III, IV with respect to the background $G\cap D$}.

 \item
 
Assume that the background $G \cap D$ of $\Gamma \cap D$ is as in the left figure of Figure \ref{Fig3} (II): 
we have a circle component $C$ such that 
when we consider only the background $G\cap D$, $C$ is the boundary of a block $R$ of $G$ with nonzero label $\epsilon i$, where $\epsilon$ is the sign of overlapping layers, and 
 $i$ is a positive integer. Further, assume that in the neighborhood of $R$, the complement of $R$, that is a block of $G$, has the  label $\epsilon (i-1)$. Here, we do not give a condition concerning the number of dots in the circle component: it can have dots, and we also include the case it has no dots; see Figure \ref{Fig3} (II). 
 
Then, remove the circle component $C$. 

\item
Assume that the background $G\cap D$ of $\Gamma \cap D$ is as in the left figure of Figure \ref{Fig3} (III): we have 
a loop component $C$ which is the boundary of an embedded disk $E$ with one non-smooth point such that $E$ is a block of $G \cap D$ with a nonzero label  $\epsilon i$, where $\epsilon$ is the sign of overlapping layers, and $i$ is a positive integer. Further, we assume that in the neighborhood of the disk $E$, the block of $G\cap D$ in the complement whose closure's boundary contains $C$ has the label $\epsilon (i-1)$. 

The background $G \cap D$ is the union of $C$ and a pair of adjacent arcs $a_1$ and $a_2$ of $G$ with/without dots. 
Let $G'\cap D$ be the dotted diagram obtained from $(G\cap D) \backslash C$ by taking the closure and smoothing the non-smooth point; further, if $C$ has dots, then we add a dot at the joining point of $a_1$ and $a_2$. 
Then, we deform $G\cap D$ to $G' \cap D$ with overlapping layers unchanged; see Figure \ref{Fig3} (III). 

\item
A local move such that the background $G\cap D$ is as illustrated in Figure \ref{Fig3} (IV), where we require that there is a dot on each arc, $\epsilon$ is the sign of overlapping layers, $i$ is a positive integer, and the arcs admit induced orientations. The overlapping layers are unchanged. Further we require that the resulting graph is a dotted diagram. 

We call the part of the block of $G \cap D$ before the deformation with the label $\epsilon i$ the {\it middle block}. 
Let $p_1$ and $p_2$ (respectively $a_1$ and $a_2$) the dots (respectively the arcs) of $G\cap D$. We call $p_1$ and $p_2$ (respectively $a_1$ and $a_2$) the {\it involved dots} (respectively {\it involved arcs}). We call a deformation IV a deformation IV {\it between $p_1$ and $p_2$} or {\it between $a_1$ and $a_2$}. 
We remark that a deformation IV is not uniquely determined. 
\end{enumerate}

In order to determine deformations IV with respect to a fixed background, we consider local moves $\mathcal{E}$, as follows. 
We say that $\Gamma'$ is obtained from $\Gamma$ by 
a {\it local move} or a {\it deformation} $\mathcal{E}$ if $\Gamma$ is changed to $\Gamma'$ by a local move in an embedded disk $D\subset \mathbb{R}^2$, satisfying the following. 
We have the situation (1) and (2), and the assumption of deformations II, III, IV,  and the background $G \cap D$ consists of an arc such that the blocks have the labels $\epsilon (i-1)$ and $\epsilon i$ for some $i>0$ as in the left figure of Figure \ref{Fig3} (I), 
where we do not give a condition for the number of dots. 
Then, we define a local move $\mathcal{E}$ by the move that deforms $G \cap D$ to another background $G'\cap D$ by moving the arc 
 by an ambient isotopy of $D$, under the condition that the result does not create loop components. The overlapping layers are unchanged, and we require that the resulting graph is a dotted diagram. 

\begin{enumerate}

\item[($\mathcal{E}$)] 
In the above situation, move an arc with/without dots of the background by an ambient isotopy of $D$, with overlapping layers unchanged, under the condition that the result does not create loop components. 
In particular, move a dot in the arc ignoring overlapping layers. Further we require that the resulting graph is a dotted diagram. 
\end{enumerate}

\end{definition}

\begin{figure}[ht]
\centering
\includegraphics*[height=5cm]{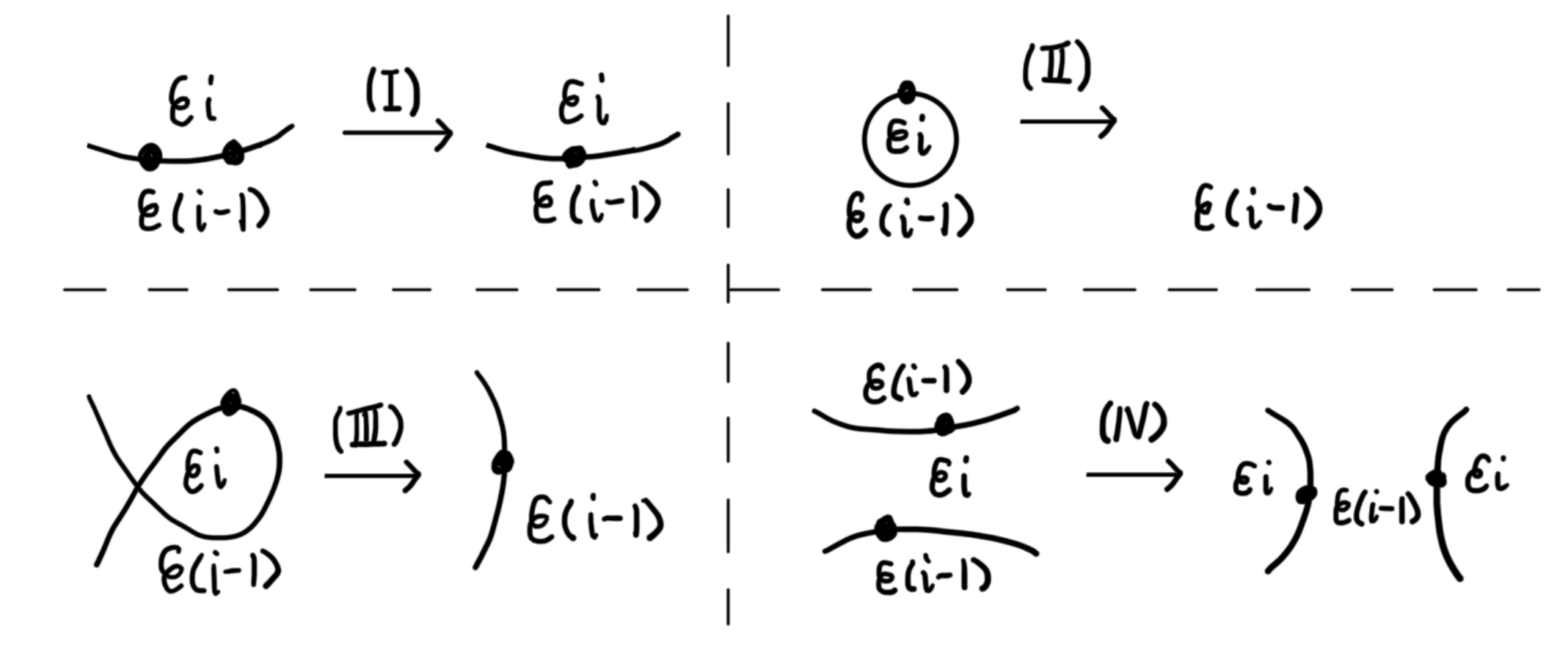}
\caption{Local deformations I--IV, where the figures illustrate backgrounds for deformations II, III, IV. Here, $\epsilon \in \{+1, -1\}$ and $i$ is a positive integer, and we omit the orientations of the arcs and some of the labels of the regions/blocks. A deformation IV is  applicable when the arcs admit induced orientations. 
Deformations II, III, IV can be applied including the case when the blocks are overlapped by several layers such that each overlapping layer has the label $\epsilon$.}  
\label{Fig3}
\end{figure}

By definition, we have the following lemma. 
\begin{lemma}
Deformation (IV) is well-defined up to local moves 
$\mathcal{E}$. 
\end{lemma}

\begin{remark}
Deformations I--IV are associated with transformations of a lattice diagram along a rectangle as in Figure \ref{B-fig17}. 
\end{remark}

\begin{figure}[ht]
\centering
\includegraphics*[height=4cm]{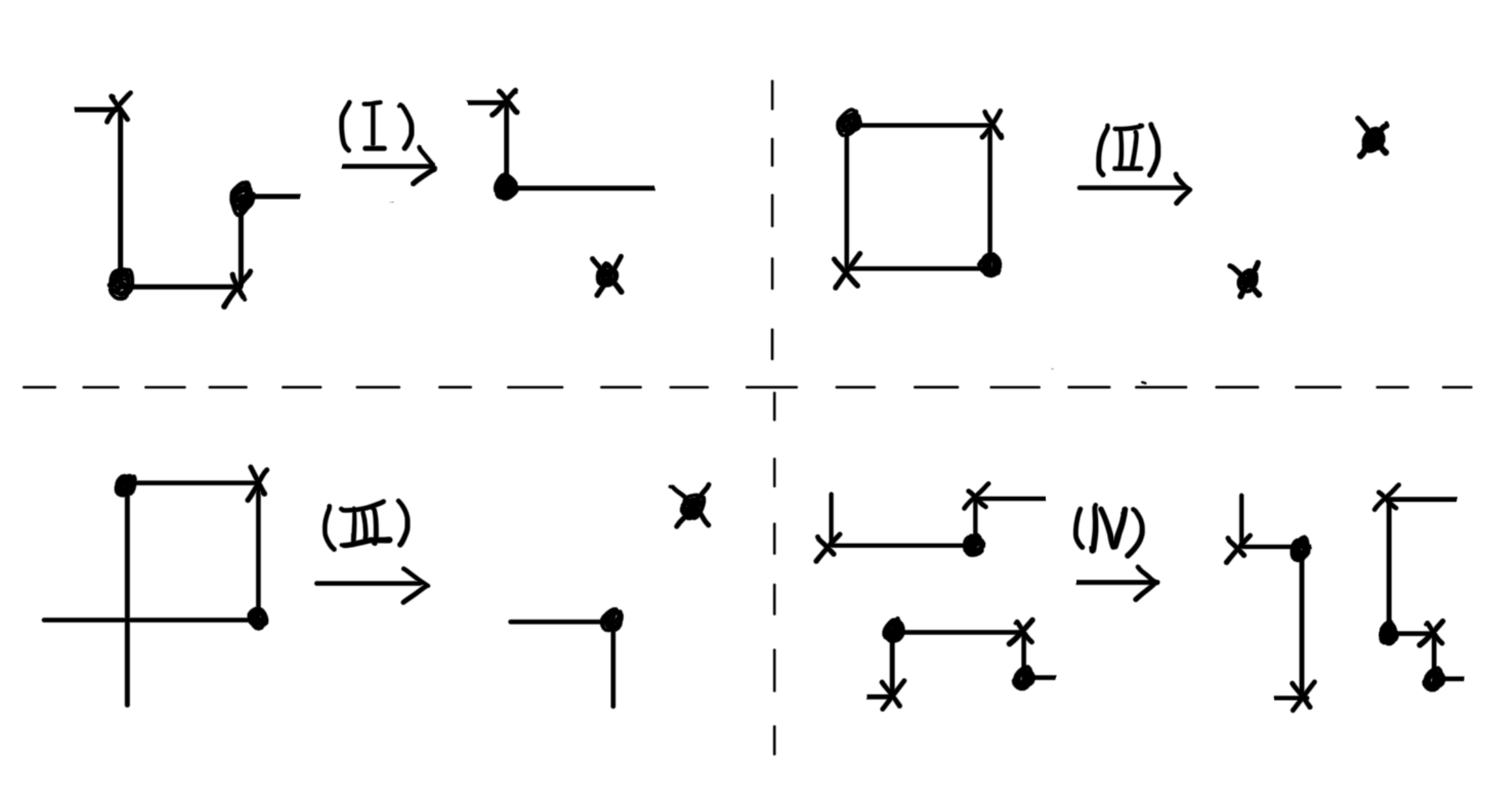}
\caption{Transformations of a lattice diagram along a rectangle corresponding to deformations I--IV, where we omit orientations of edges and labels of regions.}
\label{B-fig17}
\end{figure}

\begin{definition}\label{def3-3}
Let $\Gamma$ be a dotted diagram. We say that $\Gamma$ is  {\it reducible} if we can apply one of deformations I--IV to $\Gamma$. 
And we call a dotted diagram $\Gamma'$ a {\it reduced diagram} of 
$\Gamma$ if $\Gamma'$ is not reducible and $\Gamma'$ is obtained from $\Gamma$ by a finite sequence of deformations I--IV.

Further, we consider a specific deformation IV, called a {\it deformation IVa}, that satisfies one of the following conditions (a1) and  (a2). We denote by $\mathcal{R}$ a deformation IV. 

\begin{enumerate}

\item[(a1)]
The arcs involved in $\mathcal{R}$ are adjacent arcs of a crossing, where we ignore overlapping layers, such that $\mathcal{R}$ creates a loop component applicable of a deformation III. 

\item[(a2)]
The deformation $\mathcal{R}$ creates a circle component $C$ from two concentric circle components such that a deformation II is applicable to $C$. 
\end{enumerate}

We call deformations I, II, III and IVa {\it good deformations}. 
And we say that a sequence of deformations I, II, III, IVa is {\it in  good order} if it satisfies the following rule: 
\begin{enumerate}
\item[(1)]
After a deformation IVa satisfying (a1), next comes the deformation III that deletes the created loop component. 

\item[(2)]
After a deformation IVa satisfying (a2), next comes the deformation II that deletes the created circle component. 
\end{enumerate}

And we call a dotted diagram $\Gamma'$ a {\it good reduced diagram} of 
$\Gamma$ if $\Gamma'$ is not applicable of good deformations I, II, III, IVa, and $\Gamma'$ is obtained from $\Gamma$ by a finite sequence of good deformations I, II, III, IVa in good order.

\end{definition}

See \cite[Figure 6.2]{N} for an example of a non-empty reduced diagram. 

\subsection{Cores of deformations IV}

A deformation IV can be regarded as the result of a band surgery along an untwisted band \cite{Kawauchi}, as follows. For a dotted diagram $\Gamma$, we denote by $V$ the set of crossings and dots of $\Gamma$. Let $B \subset \mathbb{R}^2$ be an embedded disk such that $\Gamma \cap \partial B$ consists of a pair of arcs, denoted by $(a_1, a_2)$,  and a finite number of transverse intersection points such that the boundary points of $a_1$ and $a_2$ and the intersection points are in $\Gamma \backslash V$. We call $B$ an {\it untwisted band} or a {\it band} attaching to $\Gamma$. Assume that the pair of arcs $(a_1, a_2)$ in $\Gamma \cap \partial B$ does not contain dots. Then, the result of {\it band surgery} along $B$ is the dotted diagram obtained from 
\[
(\Gamma \cup \partial B)\backslash (a_1 \cup a_2). 
\]
by taking the closure. 

Let $\Gamma$ be a dotted diagram to which we apply a deformation IV, and let $\Gamma'$ be the resulting dotted diagram. 
We denote by $G\cap D$ (respectively $G' \cap D$) the background associated with $\Gamma$ (respectively $\Gamma'$) with respect to the deformation IV.  Let $B$ be the embedded disk that is the intersection of the block of $G \cap D$ with the label $\epsilon i$ and the block of $G' \cap D$ with the label $\epsilon (i-1)$. We assume that $B$ does not contain the pair of dots. Then, $B$ is a band attaching to $\Gamma'$, and the deformation IV is the result of the band surgery along $B$. The {\it core} of the band $B$ is a simple arc connecting points of arcs of $\Gamma'$ that is a retraction of $B$. 
By the condition for the signs of the overlapping layers, the core is determined up to local moves $\mathcal{E}$. 
Let $p_1$ and $p_2$ be the dots associated with the deformation IV.  We often denote the core of the band $B$ by a simple arc in the middle block  of the background connecting $p_1$ and $p_2$. When the middle block of the background has overlapping layers, we have choices of the band and the core, but the result of the deformation IV is determined up to local moves $\mathcal{E}$.

For a dotted diagram $\Gamma$, we consider the following condition. 
Let $p_1$ and $p_2$ be a pair of dots of $\Gamma$, between which a deformations IV is applicable. 
We say that $\Gamma$ {\it satisfies the condition $\mathrm{(A)}$ with respect to $(p_1, p_2)$} if it satisfies the following condition: 
\begin{enumerate}
\item[(A)]
Let $\rho$ and $\rho'$ be cores of bands connecting $p_1$ and $p_2$, 
and let $\Gamma'$ (respectively $\Gamma''$) be the dotted diagram obtained from $\Gamma$ by the deformation IV associated with $\rho$ (respectively $\rho'$). 
Then, $\Gamma'$ and $\Gamma''$ are related by local moves $\mathcal{E}$, for all possible cores $\rho$ and $\rho'$. 
\end{enumerate}

And we say that $\Gamma$ {\it satisfies the condition} (A) if it satisfies the condition (A) with respect to all pairs of dots. 
Further, we say that a sequence of deformations {\it satisfies the condition} (A) if any appearing dotted diagram satisfies the condition (A). 

We remark in an imprecise expression that if the \lq\lq middle block'' contains connected components that cannot be regarded as the subdiagram coming from overlapping layers, then the dotted diagram does not satisfy the condition (A).  
 
\begin{theorem}\label{prop3-5}
Let $\Gamma$ be a dotted diagram, and let $\Gamma'$ be a reduced diagram of $\Gamma$ whose associated sequence of deformations  satisfies the condition $\mathrm{(A)}$, and the sequence does not contain a deformation IV applied between a circle/loop component $C$ applicable of a deformation II/III and an arc of an overlapping layer of $C$. 
Then, $\Gamma'$ is uniquely determined up to local moves $\mathcal{E}$ and deformations I. 
\end{theorem}

Let $\mathcal{R}$ be a deformation consisting of deformations I--IV satisfying the following property. 
\begin{enumerate}
\item[(X)]
Let $\Gamma$ be a dotted diagram and let $\Gamma'$ be the dotted diagram obtained from $\Gamma$ by $\mathcal{R}$. Then, the number of dots of $\Gamma'$ is equal to or less than that of $\Gamma$, and the number of crossings of $\Gamma'$ is equal to or less than $\Gamma$, for any dotted diagram $\Gamma$ applicable of $\mathcal{R}$. 
\end{enumerate}
We call a deformation satisfying property (X) a {\it deformation X}. Then we have the following theorem; see Remark \ref{rem920}. 

\begin{theorem}\label{r-prop3-5}
For a dotted diagram $\Gamma$, any sequence of deformations consisting of deformations X is finite. In particular, there exists a good reduced diagram of $\Gamma$.
\end{theorem}

\section{Proofs of Theorems \ref{prop3-5} and \ref{r-prop3-5}}\label{sec-proof}

We modify the proof Lemma 6.2 in \cite{N}. See also Remark  \ref{rem902}.

\begin{proof}[
Proof of Theorem \ref{prop3-5}]
By Lemma \ref{rem915}, 
it suffices to show that 
(1) when we have diagrams $\Gamma_1$ and $\Gamma_2$ obtained from  $\Gamma$ by a finite sequence of deformations I--IV and local moves $\mathcal{E}$ satisfying the condition (A), $\Gamma_1$ and $\Gamma_2$ can be deformed by deformations I--IV and local moves $\mathcal{E}$ to the same dotted diagram, with the assumption of Theorem \ref{prop3-5}. 
If there are two dots $p_1$ and $p_2$ on an arc $\alpha$, we can apply a deformation IV between $p_1$ and $p_2$. By the condition of signs for the overlapping layers of the middle middle block and the condition (A), we see that any deformation IV between $p_1$ and $p_2$ deforms $\alpha$ into an arc $\alpha'$ and a circle component $C$, where $\alpha'$ is an arc with a dot obtained from  $\alpha$ by a deformation $\mathcal{E}$ and a deformation I, and $C$ is a circle component with a dot applicable of a deformation II; we remark that by the former condition, we see that any core associated with the deformation IV, which is an arc connecting $p_1$ and $p_2$ with no self-intersections, does not intersect with $\alpha$ except at the endpoints $p_1$ and $p_2$. 
And if there is a pair of arcs $\alpha, \alpha'$ with dots where deformations IV are applicable, 
then, no matter what times we apply deformations IV, the resulting dotted diagrams can be deformed to the same dotted diagram as when we consider $\alpha, \alpha'$ with one dot on each arc; see the left figure of Figure \ref{B-fig11}. 
Similarly, if there is an arc $\alpha$ with dots and several arcs where deformations IV are applicable between $\alpha$ and the other arcs, 
then, no matter what times we apply deformations IV between $\alpha$ and the other arcs, the resulting dotted diagrams can be deformed to the same dotted diagram as when we consider $\alpha$ with one dot; see the right figure of Figure \ref{B-fig11}. 
Further, deformations II and III are not effected by the number of dots. 
So we can assume that there is at most one dot in each arc. 
 
If we have arcs where both deformations II and IV are applicable, then, since by assumption we do not have a deformation IV applied between a circle component $C$ applicable of a deformation II and an arc of an overlapping layer of $C$, the resulting dotted diagrams can be deformed to the same dotted diagram; see Figure \ref{B-fig12}. We remark that we do not have a deformation IV between the arc used in the local move $\mathcal{E}$ and an arc of its overlapping layer. 
And by Lemma \ref{lem913}, there do not exist arcs where both deformations III and IV are applicable. 

When a deformations II or III can be applied, and the interior of the embedded disk $D$ whose boundary is the circle/loop component of the deformations II or III contains arcs where a deformation IV is applicable, by the condition of the signs of the overlapping layers of a deformation IV, the deformation II or III is applicable after the deformation IV. 
Further, by Lemma \ref{lem912}, when we have a circle/loop component $C$ applicable a deformation II/III  such that the interior of the disk with the boundary $C$ contains another circle/loop component $C'$ applicable of a deformation II or III, the result of the deformations II/III to $C$ and then $C'$ is the same with that of the deformations II/III to $C'$ and then $C$, up to local moves $\mathcal{E}$. 
Hence we see the following. We consider a dotted diagram $\Gamma$ that deforms to $\Gamma'$ by a deformation II or III. 
We denote by $G$ the circle/loop components and arcs of $\Gamma$ which are applicable of deformations I--IV in $\Gamma$ and $G \cap \Gamma'=G$. Then $G$ in $\Gamma'$ are applicable of the same deformations I--IV, where two deformations are the {\it same} if they are the same as deformations of diagrams; we remark that the labels of the regions might differ.  

Further, if there are plural circle/loop components applicable of deformations II or III that have overlapping with each other, since the regions of loop/circle components have all positive/negative labels, we can delete all components applying deformations II or III, and the result is independent of the order of deformations. 

Next we consider deformations IV. 
By taking the middle block to be sufficiently thin, we assume that the middle block does not contain circle/loop components. 
We recall that the sequence of deformations satisfies the condition (A). 
 
We consider the case when we have several choices of deformations IV. 
Suppose that there are three dots $p_1, p_2, p_3$ and we 
have two choices of applying a deformation IV: between $(p_1, p_2)$ or $(p_1, p_3)$. Then we can apply a deformation IV also to the other pair of dots $(p_2, p_3)$, and the result of each choice can be deformed to the same dotted diagram (up to local moves $\mathcal{E}$), as shown in the top and middle rows of figures in Figure \ref{C-fig2};  
 we remark that in the middle figure of the top row in Figure \ref{C-fig2}, we cannot apply a deformation IV to a pair of dots other than given in the figure, by seeing the labels of the regions. 

Thus, we see that if we have several possible parts of a dotted diagram $\Gamma$ to each of which a deformation I--IV, denoted by $\mathcal{R}_i$, is applicable, then, for any $i$, the dotted diagram $\Gamma'_i$ obtained by $\mathcal{R}_i$ can be deformed to the same dotted diagram by deformations I--IV and $\mathcal{E}$. 
This implies (1), and we have the required result. 
\end{proof}

\begin{figure}[ht]
\centering
\includegraphics*[height=4.5cm]{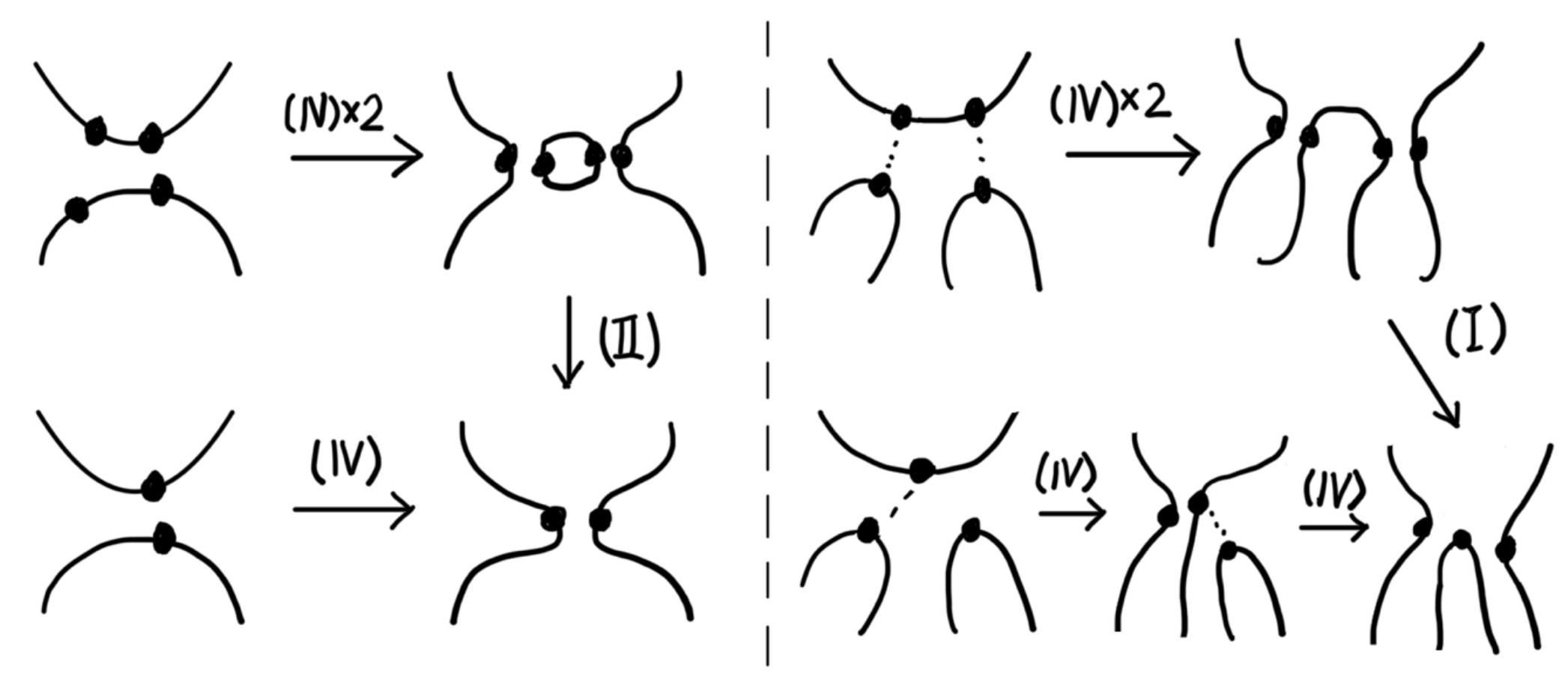}
\caption{If there is a pair of arcs $\alpha, \alpha'$ with dots where deformations IV are applicable, 
then, no matter what times we apply deformations IV, the resulting dotted diagrams can be deformed to the same dotted diagram as when we consider $\alpha, \alpha'$ with one dot on each arc (left figure). If there is an arc $\alpha$ with dots and several arcs where deformations IV are applicable between $\alpha$ and the other arcs, 
then, no matter what times we apply deformations IV between $\alpha$ and the other arcs, the resulting dotted diagrams can be deformed to the same dotted diagram as when we consider $\alpha$ with one dot (right figure). We omit orientations of arcs and labels of regions.} 
\label{B-fig11}
\end{figure}

\begin{figure}[ht]
\centering
\includegraphics*[height=6.5cm]{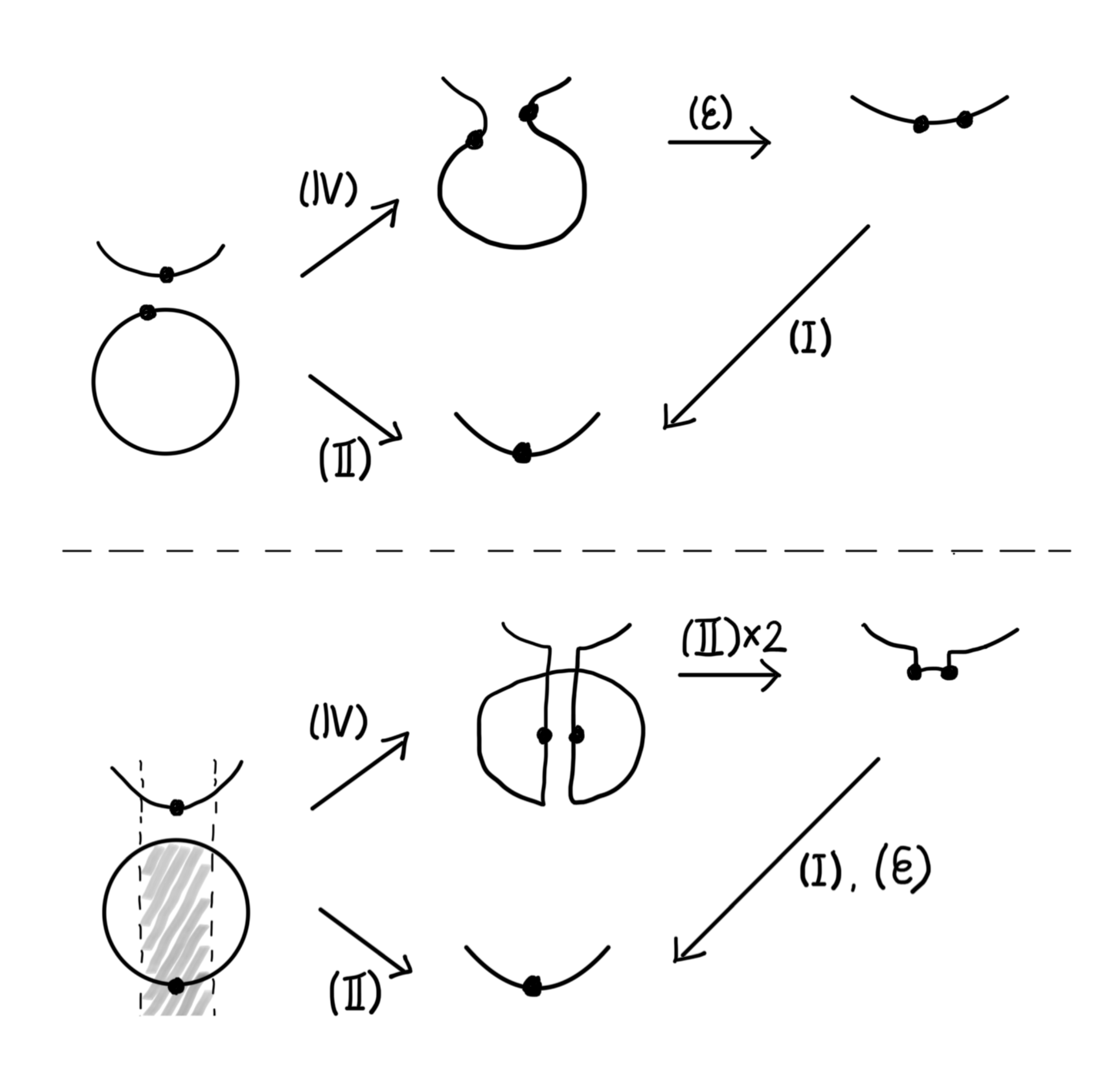}
\caption{If we have arcs where both deformations II and IV are applicable, then the resulting dotted diagrams can be deformed to the same dotted diagram, where we omit the orientations of arcs and labels of regions; we remark that in the lower figure, the shadowed area is the overlapping layer. We also remark that the upper figure is \cite[Figure 6.4]{N}.}
\label{B-fig12}
\end{figure}

\begin{figure}[ht]
\centering
\includegraphics*[height=7cm]{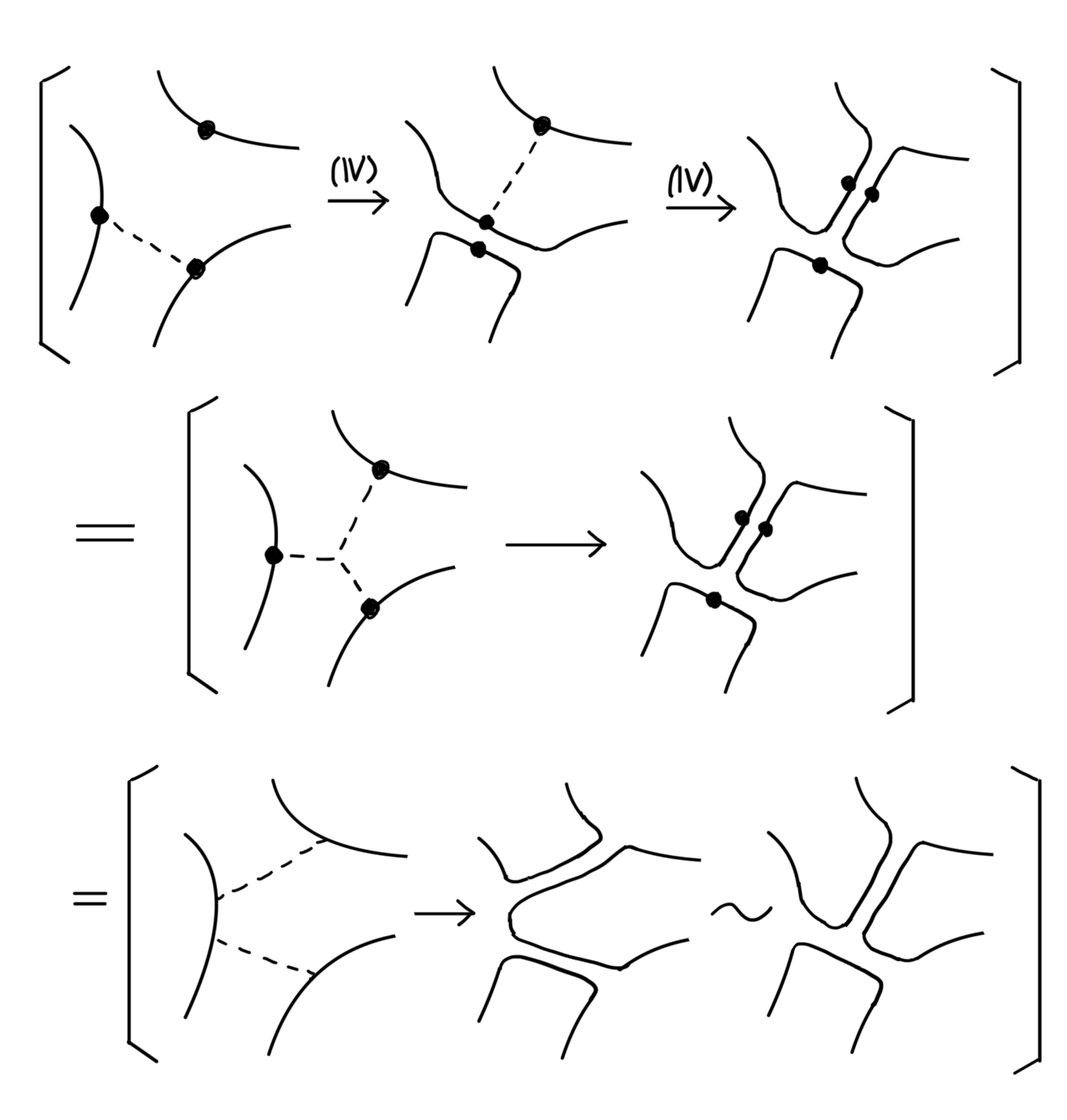}
\caption{A deformation IV can be regarded as the result of band surgery along an untwisted band. In the top and the bottom rows of figures, we denote by dotted arcs cores of bands. We assume that each appearing dotted diagram satisfies the condition (A) (see Theorem \ref{prop3-5}). 
If there are $n$ dots between which there are several possible sequences of $n-1$ deformations IV, then the result is independent of the choice of sequence up to local moves $\mathcal{E}$ (the middle row). The deformation sequence can be regarded as band surgery along mutually disjoint bands such that the endpoints of cores are in the neighborhoods of dots, where we omit dots in the figures (the bottom row). }
\label{C-fig2}
\end{figure}

\begin{proof}[
Proof of Theorem \ref{r-prop3-5}]

Let $\Gamma=\Gamma_0, \Gamma_1, \Gamma_2, \ldots, \Gamma_n$ be dotted diagrams that appear in a sequence of deformations X. 
We denote by $\mathcal{R}_i$ the deformation such that $\Gamma_{i+1}$ is obtained from $\Gamma_i$ by $\mathcal{R}_i$. 
 
When $\mathcal{R}_i$ is a deformation I, then the number of dots in ${\Gamma}_{i+1}$ is less than that of ${\Gamma}_{i}$, and the number of crossings of ${\Gamma}_{i+1}$ equals that of ${\Gamma}_{i}$. 
When $\mathcal{R}_i$ is a deformation III, then the number of dots of  ${\Gamma}_{i+1}$ is equal to or less than that of ${\Gamma}_{i}$, and the number of crossings of ${\Gamma}_{i+1}$ is less than that of ${\Gamma}_{i}$. 
Thus, in any sequence of deformations I--IV, the number of deformations I and III is finite. 

Since a deformation II decreases the number of circle components, 
if there exists a sequence consisting of deformations II, the number of deformations is finite. 
And if a deformation IV creates a circle component applicable of a deformation II, then the circle component has a dot; hence the resulting graph can be deformed by a deformation II to a dotted diagram the number of whose dots is less than that of the initial dotted diagram, and the number of whose crossings is equal to or less than that of the initial dotted diagram. And we have a similar situation if a deformation III creates a circle component applicable of a deformation II. Thus, in any sequence of deformations I--IV, the number of deformations II is also finite. 

Hence, since a deformation X consists of deformations I--IV, 
it suffices to show that for an arbitrary dotted diagram $\Gamma$, there does not exist a sequence of deformations consisting of 
infinite number of deformations IV. 
This is shown in Lemma \ref{lem826}. Thus, for a dotted diagram $\Gamma$, any sequence of deformations X applicable to $\Gamma$ is finite. 

Now we show that good deformations I, II, III, and IVa satisfy the property (X). We see that deformations I, II, and III satisfy the property (X). From now on, we show that a good deformation IVa satisfies the property (X). 
We call a deformation IVa satisfying the condition (a1) (respectively (a2)) a deformation IV (a1) and (respectively IV (a2)). For a deformation IV (a1), we take the core sufficiently near the adjacent edges of the crossing. 
When $\mathcal{R}_i$ is a deformation IV (a1), then, the number of dots of $\Gamma_{i+2}$ is equal to or less than that of ${\Gamma}_{i}$, and the number of crossings of $\Gamma_{i+2}$ 
is less than that of $\Gamma_i$. When $\mathcal{R}_i$ is a deformation IV (a2), then, the number of dots of $\Gamma_{i+2}$ is less than that of ${\Gamma}_{i}$, and the number of crossings of $\Gamma_{i+2}$ 
is less than or equal to that of $\Gamma_i$. 
Thus a deformation IVa is a deformation X. 
Hence, for a dotted diagram $\Gamma$, there exists a good reduced diagram of $\Gamma$, which is the required result. 
\end{proof}

\begin{remark}\label{rem920}
We remark that by the same argument in the proof of Theorem \ref{r-prop3-5}, we can see the following. 
We denote by a deformation IVb a deformation IV that does not have overlapping layers. 
Then, for a dotted graph $\Gamma$, we can construct a sequence of deformations consisting of a finite number of deformations I, II, III, and IVb, so that the result is a dotted diagram to which we cannot apply deformations I, II, III, and IVb. 
\end{remark}

\section{Deformations of admissible dotted diagrams and dissolutions of lattice diagrams}\label{deformation}

In this section, we investigate relation between deformations of admissible dotted diagrams and transformations of lattice diagrams. We show that if an admissible dotted diagram is deformed to the empty graph by a sequence of good deformations in good order, then its associated lattice diagram has a dissolution with minimal area, and we obtain deformations of dotted diagrams describing a dissolution of lattice diagrams with minimal area (Theorem \ref{thm3-7}). 
Further, we show that if a dotted diagram has \lq\lq many'' dots, then it is admissible (Proposition \ref{prop3-8}), and a dotted diagram with ``many'' dots is reduced to the empty graph (Proposition \ref{prop3-9}). 

 When we have a lattice diagram $P$ with initial vertices $\mathrm{Ver}_0(P)$ and terminal vertices $\mathrm{Ver}_1(P)$, 
we have given a transformation along a rectangle $R$ that has vertices in $\mathrm{Ver}_0(P)$ as its diagonal vertices. 
The transformation of $P$ along the rectangle $R$ is 
the transformation from $P$ to the lattice diagram whose initial vertices  are $t(\mathrm{Ver}_0(P); R)$ and terminal vertices are $\mathrm{Ver}_1(P)$. 
We call such a transformation of $P$ a {\it normal transformation}. 
Now, we define a reversed transformation as follows. 
\begin{definition}\label{def818}
Let $P$ be a lattice diagram, and let $R$ be 
a rectangle that has vertices in $\mathrm{Ver}_1(P)$ as its diagonal vertices. 
Then, we define the {\it reversed transformation} of $P$ along the rectangle $R$ as
the transformation from $P$ to the lattice diagram whose initial vertices  are $\mathrm{Ver}_0(P)$ and terminal vertices are $t(\mathrm{Ver}_1(P); R)$. 
\end{definition}

\begin{theorem}\label{thm3-7}
Let $\Gamma$ be an admissible dotted diagram associated with a lattice diagram $P$. 
Then, 
if $\Gamma$ has a good reduced diagram that is the empty graph, then there exists a dissolution of $P$ with minimal area, that is, a dissolution satisfying the equality of (\ref{eq2}). 
\end{theorem}

\begin{proof}
Suppose that $\Gamma$ has a good reduced diagram that is the empty diagram. 
By an argument as in the proof of \cite[Theorem 5.9]{N}, for each deletion of a circle/loop component, there exists a corresponding dissolution of the lattice diagram realizing the minimal area; here we remark that in the case of a loop component, the dissolution of the lattice diagram might consists of normal transformations and reversed transformations.

When we have a deformation IVa satisfying the condition (a2) that creates a circle component $C$ from two concentric circle components and then a deformation II that deletes $C$, the lattice diagram satisfies the condition (2) in  \cite[Theorem 5.9]{N}, and hence there exists a transformation of the lattice diagram that realizes the deformations II and III. 
For each deformation IVa satisfying the condition (a1) that is applied between adjacent arcs of a crossing, there exists a normal/reversed transformation of the lattice diagram that realizes the deformation IVa by Lemma \ref{lem-822}, that is the transformation as shown in Figure \ref{F-fig1}. We denote the resulting lattice diagram by $P$. Then, as shown in Figure \ref{F-fig1}, $P$ is not the result of the transformation of the dotted diagram, denoted by $G$, but 
since the deformation sequence is in good order, we delete the loop component of $G$ by a deformation III, and the result is the dotted diagram associated with $P$. 

Hence, by Lemma \ref{lem-821}, we have a dissolution of the lattice diagram with minimal area, consisting of only normal transformations. Thus we have the required result. 
 \end{proof}

\begin{figure}[ht]
\centering
\includegraphics*[height=4cm]{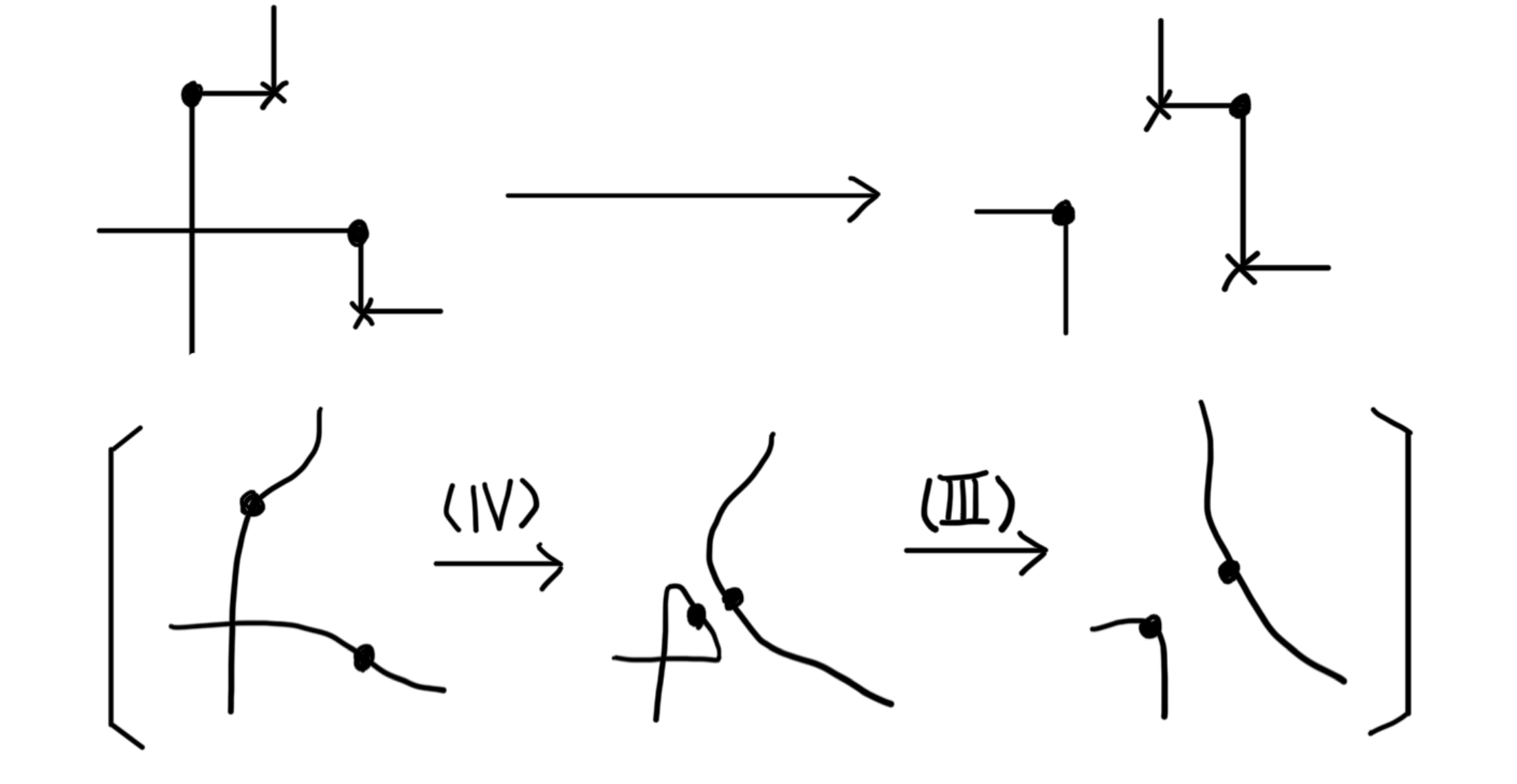}
\caption{Transformations of a lattice diagram along a rectangle   and the corresponding sequence of deformations IVa and III in good order, where we omit orientations of edges and labels of regions.}
\label{F-fig1}
\end{figure}

A dotted diagram is admissible when it has \lq\lq many'' dots. 

\begin{proposition}\label{prop3-8}
Let $\Gamma$ be a dotted diagram such that each arc has at least two dots. Then, $\Gamma$ is admissible. 
\end{proposition}

\begin{proof}
For a lattice diagram $P$, we call the part of $P$ corresponding to an arc of its associated dotted diagram an {\it arc} of $P$. 

When we consider an arc of a lattice diagram with at least two initial vertices, we can construct an arc $\alpha$, connecting a given pair of distinct points, such that the edges connecting to the points are in given directions, that is, $\alpha$ consists of edges such that the initial edge is in the $a$-direction with orientation coherent/incoherent with the $a$-axis, and the terminal edge is in the $b$-direction with orientation coherent/incoherent with the $b$-axis, for any given pair $(a,b)$ and orientations ($a,b \in \{x,y\}$). Hence, by fixing each crossing of $\Gamma$ to be a crossing consisting of edges in the $x$-direction and $y$-direction, 
we can construct the corresponding lattice diagram. Thus $\Gamma$ is admissible. 
\end{proof}

\begin{proposition}\label{prop3-9}
Let $\Gamma$ be a dotted diagram such that each arc has at least one dot. Then, the reduced diagram of $\Gamma$ is the empty graph. 
Further, if $\Gamma$ is admissible, then any lattice diagram associated with $\Gamma$ has a dissolution with minimal area. 
\end{proposition}

\begin{proof}
We show that (1) we can deform $\Gamma$ by good deformations I, II, III, IVa in good order to a dotted diagram $\Gamma'$ consisting only of mutually disjoint circle components such that each circle component has at least one dot, and (2) we can deform the dotted diagram $\Gamma'$ by deformations I--IV to the empty graph, and (3) if $\Gamma$ is admissible, then, for a lattice diagram whose associated dotted diagram is $\Gamma'$, there exists a dissolution with minimal area. Then, together with the argument in the proof of Theorem \ref{thm3-7}, we have the required result. 

We show (1). 
Since each arc has a dot, 
We see that around each crossing, we have a pair of adjacent arcs between which a deformation IVa is applicable, if they do not form a loop component. 
So we apply deformations IVa for such arcs, and then we apply deformations III to delete the created loop components. Then we have a dotted diagram consisting of circle components that are mutually disjoint and each circle component has at least one dot. 

We show (2). 
From now on, we call a circle component simply a {\it circle}. 
We consider a set of concentric circles in $\Gamma'$ from an outermost circle to an innermost circle. Then, let $\rho$ be a path from the outermost region to the innermost disk such that $\rho$ crosses each concentric circle exactly once. 
Then, the labels of the regions crossed by $\rho$ form a sequence of integers from zero to some $\epsilon n$ for a positive integer $n$ and $\epsilon \in \{+1, -1\}$. If the sequence does not contain zero other than the initial zero, then, $n$ is not zero and the sequence ends with $\epsilon (n-1), \epsilon n$; hence we apply a deformation II and delete the innermost circle.

If the sequence contains zero other than the initial zero, then we have a subsequence $0, \epsilon 1, \epsilon 2, \ldots, \epsilon (n-1), \epsilon n, \epsilon (n-1), \ldots, \epsilon 2, \epsilon 1, 0$, for a positive integer $n$ and  $\epsilon \in \{+1, -1\}$. Let $C_1$ and $C_2$ be the concentric circles  whose arcs bounds the region with the label $\epsilon n$ such that $C_1$ is the outer circle. We denote by $R$ the region with the label $\epsilon n$. 
The boundary of the closure of $R$, denoted by $\partial R$, consists of $C_1$, $C_2$ and several circles, denoted by $C_3, \ldots, C_m$. 

\noindent
(Case 1)
If the circles other than $C_1$ have the same orientations, then the regions adjacent to $R$ in the disks with boundaries  $C_2, \ldots, C_m$ have the label $n-1$, and we apply deformations IV between $C_1, \ldots, C_m$ to make $\partial R$ into one circle and then we apply a deformation II to delete $R$.

\noindent
(Case 2)
If there exists a circle $C_j$ $(j=1, \ldots, m)$ with orientation opposite to that of $C_2$, then, the region adjacent to $R$ in the disk with the boundary $C_j$ has the label $\epsilon(n+1)$, and we take a new path $\rho$ from $R$ to an innermost circle in the disk with the boundary $C_j$. If the sequence of the labels does not contain $\epsilon n$ other than the initial $\epsilon n$, then, we can delete the innermost circle by a deformation II. And if the sequence of the labels contains $\epsilon n$ other than the initial $\epsilon n$, then, we consider the circles which are the boundaries of the region the absolute value of whose label is the largest, and repeat the same process as above. 
Since the circles of $\Gamma'$ are finite, 
by repeating this process, we can delete circles by the method described in (Case 1) or by a deformation II, until we have the empty graph. 

We show (3). 
Suppose that $\Gamma'$ is admissible. 
We recall that for a lattice diagram $P$ whose dotted diagram  contains a circle component $C$ applicable of a deformation II, there exists a sequence of transformations of the lattice diagrams from $P$ to $P\backslash C$ realizing the minimal area, where we denote by the same notation $C$ the part of $P$ corresponding to the circle component $C$. 

We consider the deformation described in (Case 1) in the above argument. 
Let $P$ be a lattice diagram whose dotted diagram is $\Gamma'$. We denote by the same notation $\partial R$ the part of $P$ corresponding to $\partial R$ in $\Gamma'$. 
Then, by \cite[Theorem 5.9]{N}, there exists a sequence of transformations of the lattice diagrams from $P$ to $P\backslash \partial R$ realizing the minimal area. 
Thus, by the argument concerning (2), we see that there exists a dissolution of $P$ with minimal area. 
Thus we have the required result. 
\end{proof}

By the proof of Proposition \ref{prop3-9}, we have the following. 
\begin{proposition}
Let $\Gamma$ be a dotted diagram such that each arc has at least one dot. Then, $\Gamma$ is deformed by a sequence of good deformations I, II, III, IVa in good order to a dotted diagram $\Gamma'$ consisting only of mutually disjoint circle components such that each circle component has at least one dot.  
\end{proposition}

\section{Remark and Lemmas}\label{sec-lemma}

\begin{remark}\label{rem902}
Let $\Gamma$ be a dotted diagram. 
\begin{enumerate}[(1)]

\item
Suppose that we introduce a deformation III' that deletes a loop component with no condition for the labels of regions. Then, an application of a deformation IV to a loop component $C$ is possible where $C$ is applicable of a deformation III', and the result of a deformation IV is the result of a deformation III' and a deformation V as in Figure \ref{B-fig14}; see Figure \ref{B-fig14}, and we cannot use the argument in the proof of Theorem \ref{prop3-5}. 

\item
If we include deformations V and $\mathcal{E}$ besides deformations I--IV, there exists a dotted diagram $\Gamma$ and a sequence of deformations whose result is $\Gamma$ itself; there exists an infinite loop of deformations, and we cannot have a reduced diagram, and we cannot use the argument in the proof of Theorem \ref{r-prop3-5}. For example, see Figure \ref{C-fig4}.

\end{enumerate}
\end{remark}

\begin{figure}[ht]
\centering
\includegraphics*[height=4.5cm]{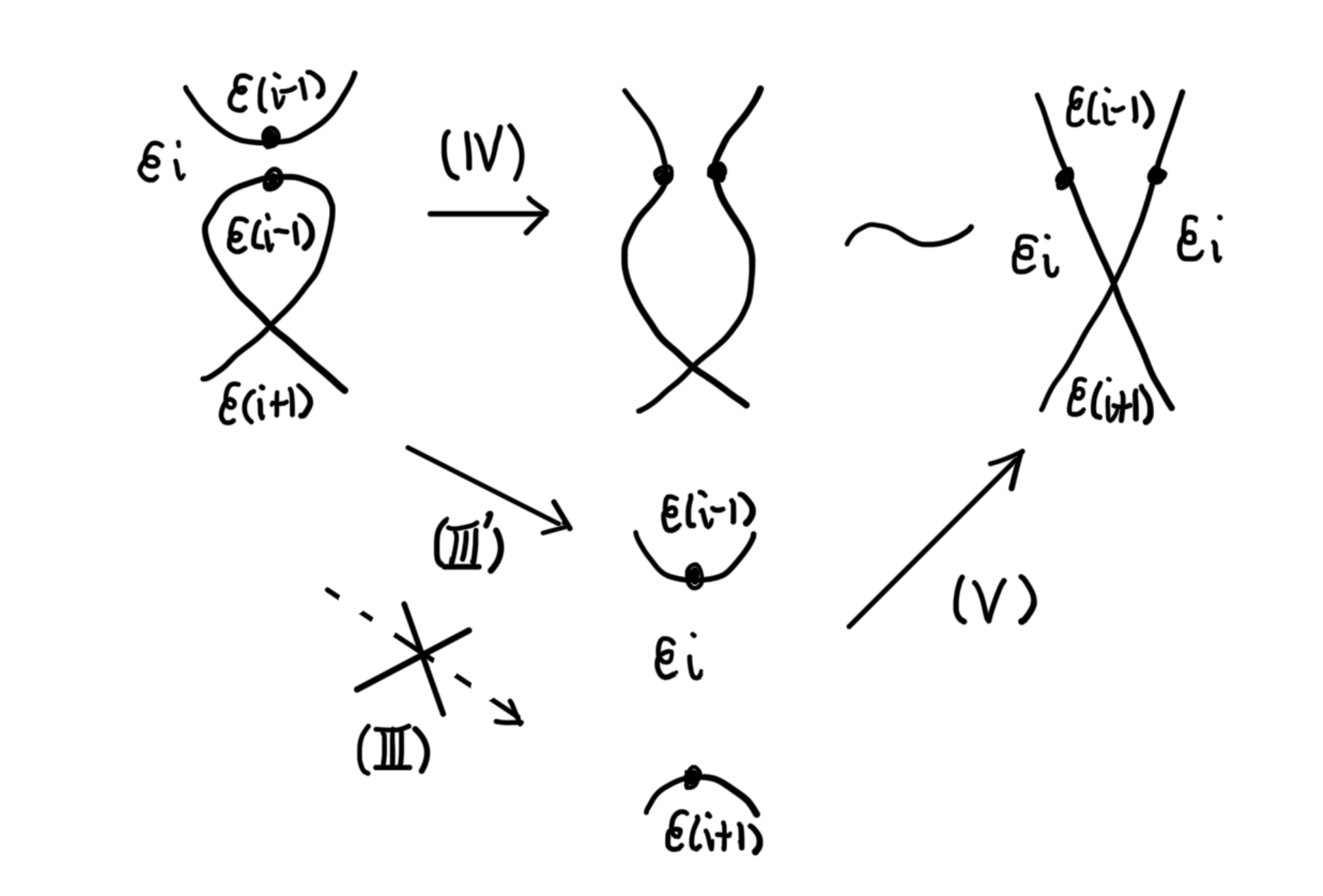}
\caption{If we have arcs where both deformations III' and IV are applicable, then the resulting dotted diagrams need a deformation V to be deformed locally to the same dotted diagram, where we omit the orientations of arcs. We remark that by the condition for labels of regions, a deformation III is not applicable. }
\label{B-fig14}
\end{figure}

\begin{figure}[ht]
\centering
\includegraphics*[height=3.5cm]{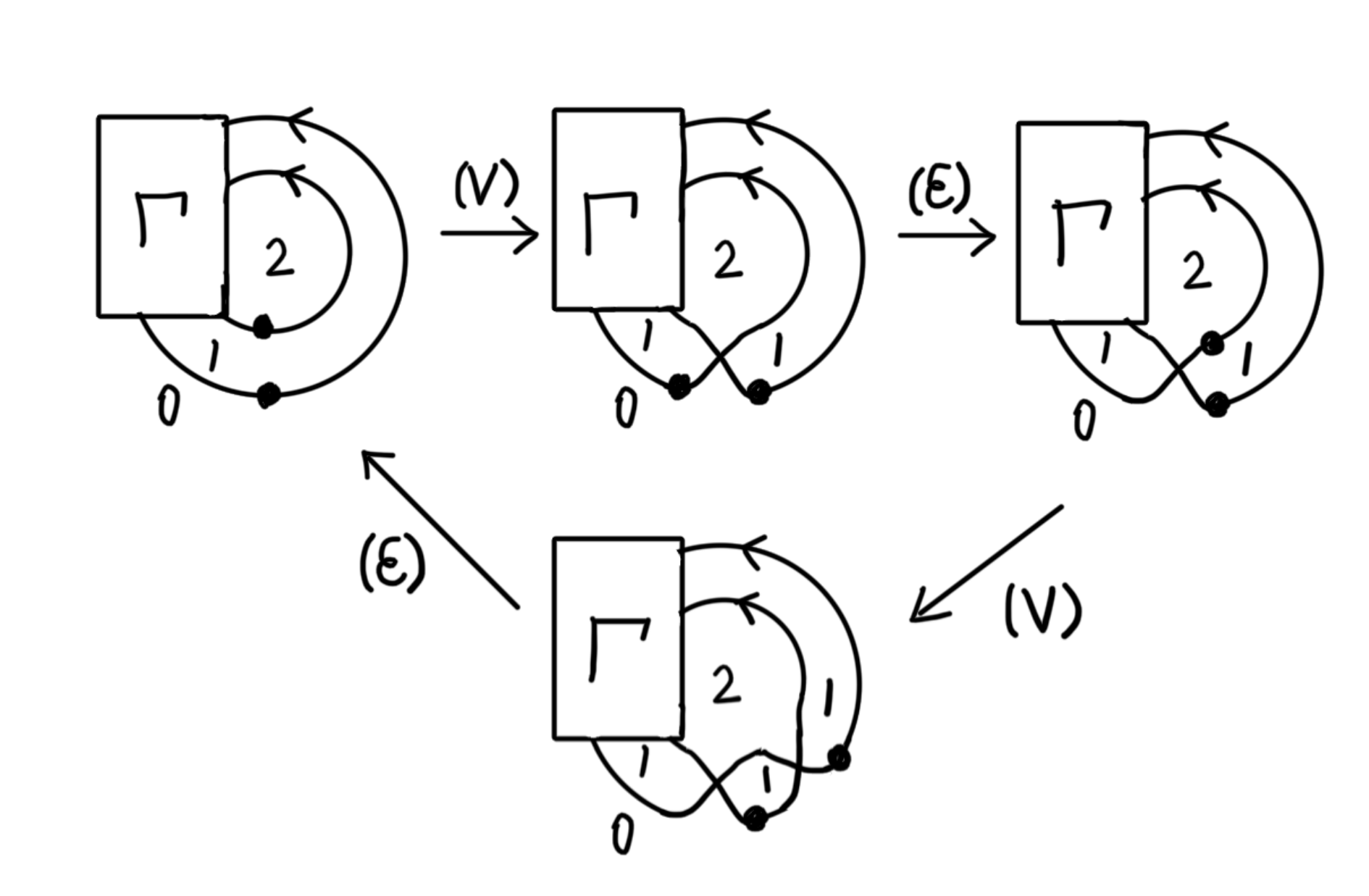}
\caption{There exist a dotted diagram $\Gamma$ admitting a sequence of deformations V and $\mathcal{E}$ whose result is $\Gamma$ itself.}
\label{C-fig4}
\end{figure}

\begin{lemma}\label{rem915}
 In the proof of Theorem \ref{prop3-5}, (1) implies the uniqueness of the reduced diagram up to local moves $\mathcal{E}$ and deformations I. 
\end{lemma}
\begin{proof}

It suffices to show that the possibilities of application of deformations II--IV are the same before and after the application of a local move $\mathcal{E}$, where we do not apply deformations IV between the arc used in the local move $\mathcal{E}$ and an arc of its overlapping layer. 
Let $\Gamma$ be a dotted diagram and let $\Gamma'$ be the dotted diagram obtained from $\Gamma$ by applying a local move $\mathcal{E}$. 
We denote the local move $\mathcal{E}$ by $\mathcal{R}$, 
and we denote the concerning arcs before and after the application by $\alpha$ and $\alpha'$, respectively.  
Further, we denote by $R_0$ the block with the label $\epsilon i$ in Figure \ref{Fig3} (I) (with/without dots) in Definition \ref{def3-2}, and we denote by $X_1, \ldots, X_m$ the overlapping layers such that each $X_i$ $(i=1,\ldots,m)$ has the sign $\epsilon$, and we put $R=R_0\cap (X_1 \cup \ldots \cup X_m)$.  

By seeing the labels of regions, we see that 
 if a deformation IV is applicable to the arc $\alpha$ (respectively $\alpha'$) and some arc $\alpha''$,  then a corresponding deformation IV between the arc $\alpha'$ and $\alpha''$ (respectively $\alpha$ and $\alpha''$) is also applicable. 
And if $\mathrm{Int}(R)$ contains a middle block applicable of a deformation IV, denoted by $D$, we can also apply a deformation IV to $D$ after the application of $\mathcal{R}$. 
Thus, 
$\Gamma'$ (respectively $\Gamma$) does not admit a deformation IV that is not induced from one of those applicable to $\Gamma$ (respectively $\Gamma'$). 
 
Further, $\mathcal{R}$ does not create a circle component. And by assumption, $\mathcal{R}$ does not create a loop component. 
Thus, $\mathcal{R}$ does not create a circle/loop component, so $\Gamma'$ does not admit deformations II/III that are not induced from those applicable to $\Gamma$. 
And if the closure of $R$ intersects a circle/loop component applicable of a deformation II/III, denoted by $C$, we can also apply a deformation II/III to $C$ after the application of $\mathcal{R}$. 
Thus, 
$\Gamma'$ (respectively $\Gamma$) does not admit a deformation II/III that is not induced from one of those applicable to $\Gamma$ (respectively $\Gamma'$).

Thus, local moves $\mathcal{E}$ do not effect deformations II, III  and IV.  
We remark that local moves $\mathcal{E}$ might change arcs with dots to one arc with dots (or move a dot into another arc with a dot). Hence 
we see that (1) in the proof of Proposition \ref{prop3-5} implies the uniqueness of the reduced diagram up to local moves $\mathcal{E}$ and deformations I. 
\end{proof}

\begin{lemma}\label{lem913}
There do not exist arcs where both deformations III and IV are applicable, under the assumption that we do not apply a deformation IV between a loop component $C$ applicable of a deformation III and an arc of an overlapping layer of $C$. 
\end{lemma}
\begin{proof}
By seeing the labels of the regions, we have the required result. 
See Figure \ref{B-fig14}. 
\end{proof}

\begin{lemma}\label{lem912}
When we have a circle/loop component $C$ applicable a deformation II/III  such that the interior of the disk with the boundary $C$ contains another circle/loop component $C'$ applicable of a deformation II/III, the result of the deformations II/III to $C$ and then $C'$ is the same with that of the deformations II/III to $C'$ and then $C$, up to local moves $\mathcal{E}$. 
\end{lemma}

\begin{proof}
Let $D$ and $D'$ be the disks whose boundaries are $C$ and $C'$, respectively. 
We see that the labels of regions in $D$ are all positive or all negative, and $C$ and $C'$ has a coherent orientations as concentric circles. Hence the result of the deformations II/III to $C$ and then $C'$ is the same with that of the deformations II/III to $C'$ and then $C$, up to local moves $\mathcal{E}$. 
\end{proof}

\begin{lemma}\label{lem826}
 Let  $\Gamma$ be a dotted diagram. Then, there does not exist a sequence of deformations consisting of 
infinite number of deformations IV.  
\end{lemma}

\begin{proof}
Assume that there is a sequence of deformations IV, such that the number of deformations is greater than an arbitrarily large number $n$. 
 
Let $\Gamma'$ be a dotted diagram where a deformation IV is applied, and let $\Gamma''$ be the resultant dotted diagram. Let $B$ be the band attaching to $\Gamma'$ such that $\Gamma''$ is the result of band surgery along $B$. 
We call arcs $\Gamma' \cap B$ {\it ends} of the band $B$ attaching to $\Gamma'$. 
We remark that a deformation IV with a fixed background is determined by the core of $B$, up to local moves $\mathcal{E}$. 

When we have a sequence of deformations IV from $\Gamma$ to some dotted diagram, we have a set of bands associated with the deformations IV. We assume that the bands are sufficiently thin. We denote by the same notations $p_1, p_2, \ldots$ the dots of each dotted diagram such that for each dot $p_i$, $p_i$ of the next dotted diagram is either the same $p_i$ or one of the involved dots of the deformation IV if it is applied between $p_i$ and some $p_j$; we remark that we have two choices of $(p_1, p_2, \ldots)$ for each deformation IV. 
Let $n_i$ be the number of deformations IV applied to $p_i$. Let $\nu$ be the number of dots $p_i$ with $n_i>0$. By changing the indices, we arrange that $p_1, \ldots, p_\nu$ are the dots with $n_i>0$ $(i=1, \ldots, \nu)$. Let $N(G)$ be the union of the bands, and let $a_i$ be the end of the band containing $p_i$ of $\Gamma$ $(i=1, \ldots, \nu)$. 
Then, the result of the deformations IV is the closure of $(\Gamma \cup \partial N(G))\backslash (a_1 \cup \cdots \cup a_\nu)$. 
The union of bands $N(G)$ is a regular neighborhood of a graph $G$ which consists of edges and vertices of degree one called {\it endpoints}, and vertices $v_i$ of degree $n_i+1$  such that the endpoints are dots $p_i$ of $\Gamma$ and $v_i$ is contained in the interior of $N(G)$ $(i=1, \ldots, N$); see the middle figure in Figure \ref{C-fig2}.

By local moves $\mathcal{E}$, we regard that the result of deformations IV is the result of band surgery along bands such that for each $p_i$, the associated cores are mutually disjoint and the endpoints are $n_i$ points in the neighborhood of the dot $p_i$; see the lower figure in Figure \ref{C-fig2}. We call the endpoints the endpoints {\it associated with $p_i$}. We remark that since each  deformation IV satisfies the condition of signs for overlapping layers, together with Lemma \ref{lem0901}, we see that the cores whose endpoints are associated with $p_i$ do not have intersecting points. 

Since we can take the number of deformations to be greater than an arbitrarily large number $n$, 
there are 
cores each of whose endpoints are associated with a fixed pair of dots. The cores are mutually disjoint, and in the neighborhood of the associated dot, near the endpoints, the cores are parallel by Lemma \ref{lem0901}. 
Since the bands have no twists, when we have $m$ such cores, the resultant dotted diagram has at least $m-2$ circle components more than before the deformations, and each circle component has a dot. 
Since the number $m$ can be made arbitrarily large, the resultant diagram has more dots than $\Gamma$. Since deformations IV do not change the number of dots, this is a contradiction. 
Thus a dotted diagram does not admit a sequence of arbitrarily large number of deformations IV. 
\end{proof}

\begin{lemma}\label{lem0901}
Let $\Gamma$ be a dotted diagram and let $p$ be a dot of $\Gamma$. Let $D$ be a small embedded disk around $p$ such that the arc of $\Gamma$ divides $D$ into two embedded disks $D_1$ and $D_2$. Let $\mathcal{X}$ be the set of  deformations  IV applicable between $p$ and other dots. Then, the middle blocks of deformations in $\mathcal{X}$ intersect 
with $D$ in one of the disks $D_1$ and $D_2$. 
\end{lemma}

\begin{proof}
Assume that we can apply a deformation IV, denoted by $\mathcal{R}_1$ (respectively $\mathcal{R}_2$), between $p$ and a dot $p_1$ (respectively $p$ and $p_2$) such that the middle block  intersects with $D$ in $D_1$ (respectively $D_2$). 
Since $\mathcal{R}_1$ is applicable, we see that when $D_1$ is not overlapped, the label of $D_1$ (respectively $D_2$) is $\epsilon i$ (respectively $\epsilon (i-1)$), where $\epsilon \in \{+1, -1\}$ and $i$ is a positive integer: 
the absolute value of the label of $D_1$ is greater than that of $D_2$. 
Further, when the middle block containing $D_1$ is overlapped by some layers, the label of each region in the middle block is $\epsilon j$ for some $j \geq i$; so the case does not occur that the absolute value of the label of a region in $D_1$ is smaller than that of $D_2$. 

By the same argument, since $\mathcal{R}_2$ is applicable, we see that the absolute value of the label of a region in $D_2$ is greater than that of $D_1$, which is a contradiction. 
Hence we have the required result. 
\end{proof}

\begin{lemma}\label{lem-822}
Let $P$ be a lattice diagram and 
let $\Gamma$ be its associated dotted diagram. 
When we can apply a deformation IVa that satisfies the condition (a1) given in Definition \ref{def3-3}, then 
there exists a normal/reversed transformation of $P$ that realizes the deformation IVa. 
\end{lemma}

\begin{proof}
Suppose that we can apply to $\Gamma$ a deformation IV in question. We recall that the condition (a1) is that the arcs involved in the deformation are adjacent arcs of a crossing.

We denote by $D$ the union of the regions of $P$ associated with the middle block. 
Then, the pair of dots involved in the deformation are on adjacent arcs connected by a crossing such that around the crossing, the other arcs are in the complement of $D$. Let $e$ and $e'$ be the pair of edges of $P$ in the $x$-direction and $y$-direction that form the crossing, and let $v$ and $v'$ be the pair of vertices that are endpoints of $e$ and $e'$ contained in the boundary of the closure of $D$. Let $R$ be the rectangle whose pair of diagonal vertices are $v$ and $v'$. We can see that $v$ and $v'$ are both initial vertices or both terminal vertices of $P$. The transformation of $P$ along $R$ is the required transformation: it realizes the deformation IV, and it is a normal (respectively reversed) transformation if $v$ and $v'$ are the initial (respectively terminal) vertices. 
\end{proof}

\begin{lemma}\label{lem-821}
Let $P$ be a lattice diagram. 
When $P$ admits a dissolution consisting of normal transformations and reversed transformations realizing minimal area, 
then there is a dissolution consisting only of normal transformations realizing minimal area. 
\end{lemma}

\begin{proof}
We denote the sequence of transformations by  
\[
(\mathcal{R}^1_1 \mathcal{R}^1_2 \cdots \mathcal{R}^1_{m_1}) (\mathcal{L}^1_1 \mathcal{L}^1_2 \cdots \mathcal{L}^1_{n_1})
(\mathcal{R}^2_1 \mathcal{R}^2_2 \ldots \mathcal{R}^2_{m_2}) \cdots (\mathcal{L}^r_1 \mathcal{L}^r_2 \ldots, \mathcal{L}^r_{n_r}), 
\]
where $\mathcal{R}_s^t$ (respectively $\mathcal{L}_s^t$) is a normal  transformation (respectively a reversed transformation) along a rectangle, and we apply transformations from left to right. 
For a reversed transformation $\mathcal{L}$ along a rectangle $R$ from a lattice diagram $P_1$ to a lattice diagram $P_2$, we denote by $\mathcal{\bar{L}}$ the normal transformation along $R$ from $P_2$ to $P_1$. Then, 
the sequence of transformations 
\begin{eqnarray*}&&
(\mathcal{R}^1_1 \mathcal{R}^1_2 \cdots \mathcal{R}^1_{m_1}) 
(\mathcal{R}^2_1 \mathcal{R}^2_2 \cdots \mathcal{R}^2_{m_2}) (\mathcal{\bar{L}}^1_{n_1} \mathcal{\bar{L}}^1_{n_1-1} \cdots \mathcal{\bar{L}}^1_{1})\\
&&
\ \cdot (\mathcal{R}^3_1 \mathcal{R}^3_2 \cdots \mathcal{R}^3_{m_3})
(\mathcal{\bar{L}}^2_{n_2}\mathcal{\bar{L}}^2_{n_2-1}\cdots \mathcal{\bar{L}}^2_{1}) \cdots\\
&&
\ \cdot (\mathcal{R}^r_1 \mathcal{R}^r_2 \cdots \mathcal{R}^r_{m_r})
(\mathcal{\bar{L}}^{r-1}_{n_{r-1}}\mathcal{\bar{L}}^{r-1}_{n_{r-1}-1}\cdots \mathcal{\bar{L}}^{r-1}_{1})
(\mathcal{\bar{L}}^r_{n_r} \mathcal{\bar{L}}^r_{n_r-1} \cdots \mathcal{\bar{L}}^r_{1})
\end{eqnarray*}
is the required dissolution of $P$. 
\end{proof}

\section*{Acknowledgements}
The author would like to thank the referee for his/her kind and very helpful comments. The author was partially supported by JST FOREST Program, Grant Number JPMJFR202U.

\end{document}